\documentclass[11pt,reqno,a4letter]{article} 

\usepackage{amsmath,amssymb,amsfonts,amscd}

\usepackage{url}
\usepackage{hyperref}
\usepackage[usenames,dvipsnames]{xcolor}
\hypersetup{
    colorlinks,
    linkcolor={black!63!darkgray},
    citecolor={blue!70!white},
    urlcolor={blue!80!white}
}
\usepackage{graphicx} 

\usepackage{ocr}
\usepackage[T1]{fontenc}

\newcommand{\hlocalref}[1]{\hyperref[#1]{\ref{#1}}}

\usepackage{datetime} 
\usepackage{cancel}
\usepackage{soul} 
\usepackage{subcaption}
\numberwithin{equation}{section} 
\numberwithin{figure}{section}
\numberwithin{table}{section}

\usepackage{tocloft}
\usepackage{framed} 

\usepackage{enumitem}
\setlist[itemize]{noitemsep,topsep=0pt,leftmargin=0.23in,label={\tiny$\blacksquare$}}

\usepackage{changepage}
\usepackage{rotating,adjustbox}
\usepackage{bigints}

\usepackage{diagbox}

\let\citep\cite

\newcommand{\cf}{cf.~} 
\newcommand{\Iverson}[1]{\ensuremath{\left[#1\right]_{\delta}}} 
\newcommand{\floor}[1]{\left\lfloor #1 \right\rfloor} 
\newcommand{\ceiling}[1]{\left\lceil #1 \right\rceil} 
 
\newcommand{\seqnum}[1]{\href{http://oeis.org/#1}{\color{ProcessBlue}{\underline{#1}}}}

\usepackage{upgreek,dsfont,amssymb}
\renewcommand{\chi}{\upchi}

\makeatletter
\newcommand*\rel@kern[1]{\kern#1\dimexpr\macc@kerna}
\newcommand*\widebar[1]{%
  \begingroup
  \def\mathaccent##1##2{%
    \rel@kern{0.8}%
    \overline{\rel@kern{-0.8}\macc@nucleus\rel@kern{0.2}}%
    \rel@kern{-0.2}%
  }%
  \macc@depth\@ne
  \let\math@bgroup\@empty \let\math@egroup\macc@set@skewchar
  \mathsurround\z@ \frozen@everymath{\mathgroup\macc@group\relax}%
  \macc@set@skewchar\relax
  \let\mathaccentV\macc@nested@a
  \macc@nested@a\relax111{#1}%
  \endgroup
}

\usepackage{MnSymbol}

\usepackage{ifthen}
\newcommand{\Hn}[2]{
     \ifthenelse{\equal{#2}{1}}{H_{#1}}{H_{#1}^{\left(#2\right)}}
}

\newcommand{\Floor}[2]{\ensuremath{\left\lfloor \frac{#1}{#2} \right\rfloor}}

\title{
       Exact formulas for partial sums of the M\"obius function expressed by 
       partial sums weighted by the Liouville lambda function
}
\author{Maxie Dion Schmidt \\
        Georgia Institute of Technology \\
        School of Mathematics
}  
\date{\small\underline{Last Revised:} \today \ @\ \hhmmsstime{} \ -- \ Compiled with \LaTeX2e} 

\usepackage{amsthm} 

\theoremstyle{plain} 
\newtheorem{theorem}{Theorem}

\newtheorem{prop}[theorem]{Proposition}
\newtheorem{lemma}[theorem]{Lemma}

\numberwithin{theorem}{section}
\newtheorem*{theorem*}{Theorem}
\newtheorem*{conjecture*}{Conjecture}

\theoremstyle{definition} 

\newtheorem{remark}[theorem]{Remark}
\newtheorem{definition}[theorem]{Definition}

\theoremstyle{remark}
\newtheorem*{remark*}{Remark}

\usepackage[compact]{titlesec}
 
\newcommand{\ManuscriptMarginSize}{0.75in}
\usepackage[total={7in, 9.5in},left=\ManuscriptMarginSize,right=\ManuscriptMarginSize,
	    tmargin=\ManuscriptMarginSize,bmargin=\ManuscriptMarginSize,headsep=0pt]{geometry}
\newcommand{\PlotFigureHorizontalScalingFactor}{0.8}

\renewcommand{\Re}{\operatorname{Re}}

\newcommand{\mathtext}[1]{\text{\rm #1}}

\usepackage{tikz}
\usetikzlibrary{shapes,arrows}

\allowdisplaybreaks 

\begin{document} 

\maketitle

\begin{abstract} 
\noindent  
The Mertens function, $M(x) := \sum_{n \leq x} \mu(n)$, is 
defined as the summatory function of the classical M\"obius function.
The Dirichlet inverse function $g(n) := (\omega+\mathds{1})^{-1}(n)$
is defined in terms of the shifted strongly additive function $\omega(n)$ that counts the 
number of distinct prime factors of $n$ without multiplicity. 
The Dirichlet generating function (DGF) of $g(n)$ is $\zeta(s)^{-1} (1+P(s))^{-1}$ 
for $\Re(s) > 1$ where $P(s) = \sum_p p^{-s}$ is the prime zeta function. 
We study the distribution of the unsigned functions $|g(n)|$ with 
DGF $\zeta(2s)^{-1}(1-P(s))^{-1}$ 
and $C_{\Omega}(n)$ with DGF 
$(1-P(s))^{-1}$ for $\Re(s) > 1$. 
We establish formulas for the average order and variance of 
$\log C_{\Omega}(n)$ and prove a central limit theorem 
for the distribution of its values on the 
integers $n \leq x$ as $x \rightarrow \infty$. 
Discrete convolutions of the partial sums of $g(n)$ with the prime counting function 
provide new exact formulas for $M(x)$. 

\bigskip\noindent
\textbf{Keywords and Phrases:} {\it M\"obius function; Mertens function; 
                                    Liouville lambda function; prime omega function; 
                                    Dirichlet inverse; prime zeta function; 
				                inversion of generalized convolutions. } \\[0.05cm] 
\textbf{Math Subject Classifications (2010):} {\it 11N37; 11A25; 11N60; and 11N64. } 
\end{abstract} 


\newpage
\renewcommand{\contentsname}{Article Index}
\setcounter{tocdepth}{2}
\tableofcontents

\newpage
\section{Introduction} 
\label{subSection_MertensMxClassical_Intro} 
\label{example_InvertingARecRelForMx_Intro}

\subsection{Definitions}

For integers $n \geq 2$, we define the strongly and 
completely additive functions, respectively, 
that count the number of prime divisors of $n$ by 
\begin{align*}
\omega(n) & := \sum_{p|n} 1, \mathtext{ and } 
\Omega(n) := \sum_{p^{\alpha} \mid\mid n} \alpha. 
\end{align*}
That is, if $n = p_1^{\alpha_1} \times \cdots p_r^{\alpha_r}$ is the 
factorization of $n$ into powers of distinct primes, then 
$\omega(n) = r$ and $\Omega(n) = \alpha_1 + \cdots + \alpha_r$. 
We use the convention that the functions $\omega(1) = \Omega(1) = 0$. 
The M\"obius function is the multiplicative function defined by 
\cite[\seqnum{A008683}]{OEIS}
\[
\mu(n) := \begin{cases} 
	1, & \text{ if $n = 1$}; \\ 
	(-1)^{\omega(n)}, & \text{ if $n \geq 2$ and $\omega(n) = \Omega(n)$ (i.e., if $n$ is squarefree);} \\ 
	0, & \text{ otherwise.}
        \end{cases}
\]
The Mertens function is defined by the partial sums 
\cite[\seqnum{A002321}]{OEIS} 
\begin{align} 
M(x) & := \sum_{n \leq x} \mu(n), \mathtext{ for } x \geq 1. 
\end{align} 
The Liouville lambda function is the completely multiplicative function 
defined for all $n \geq 1$ by $\lambda(n) := (-1)^{\Omega(n)}$ 
\cite[\seqnum{A008836}]{OEIS}. 
The partial sums of this function are defined by 
\cite[\seqnum{A002819}]{OEIS}
\begin{equation}
\label{eqn_LxSummatoryFuncDef_v1}
L(x) := \sum\limits_{n \leq x} \lambda(n), \mathtext{ for } x \geq 1. 
\end{equation}

\begin{definition}
For any arithmetic functions $f$ and $h$, we define their 
Dirichlet convolution at $n$ by the divisor sum 
\[
(f \ast h)(n) := \sum_{d|n} f(d) h\left(\frac{n}{d}\right), \mathtext{ for } n \geq 1.
\]
The arithmetic function $f$ has a unique inverse with respect to Dirichlet convolution, 
denoted by $f^{-1}$, if and only if $f(1) \neq 0$. 
When it exists, the Dirichlet inverse of $f$ satisfies 
$(f \ast f^{-1})(n) = (f^{-1} \ast f)(n) = \delta_{n,1}$, the 
multiplicative identity function with respect to Dirichlet convolution. 
\end{definition}

We define the Dirichlet inverse function \cite[\seqnum{A341444}]{OEIS} 
\begin{equation}
\label{eqn_gInvn_def_v1}
g(n) := (\omega + \mathds{1})^{-1}(n), \mathtext{ for } n \geq 1. 
\end{equation}
Th inverse function in equation \eqref{eqn_gInvn_def_v1} 
is computed recursively by applying the formula \cite[\S 2.7]{APOSTOLANUMT}
\[
g(n) = \begin{cases}
	1, & \text{ if $n = 1$; } \\ 
	-\sum\limits_{\substack{d|n \\ d> 1}} \left(\omega(d) + 1\right) g\left(\frac{n}{d}\right), & 
	\text{ if $n \geq 2$. }
        \end{cases}
\]
The function $|g(n)| = \lambda(n) g(n)$ denotes the absolute value of $g(n)$ 
(see Proposition \hlocalref{prop_SignageDirInvsOfPosBddArithmeticFuncs_v1}). 
The summatory function of $g(n)$ is defined as follows 
\cite[\seqnum{A341472}]{OEIS}: 
\begin{equation}
\label{eqn_GInvx_PartialSumForms_v1} 
G(x) := \sum_{n \leq x} g(n) = \sum_{n \leq x} \lambda(n) |g(n)|, \mathtext{ for } x \geq 1. 
\end{equation} 

\subsection{Statements of the main results}

\subsubsection{Partial summation identities for the Mertens function}

\begin{definition}
Let the partial sums of the 
Dirichlet convolution $r \ast h$ be defined by the function 
\begin{align*} 
S_{r \ast h}(x) & := \sum_{n \leq x} \sum_{d|n} r(d) h\left(\frac{n}{d}\right), 
	\mathtext{ for } x \geq 1. 
\end{align*}
\end{definition}

Theorem \hlocalref{theorem_SummatoryFuncsOfDirCvls} 
is proved by matrix methods in 
Appendix \hlocalref{Section_PrelimProofs_Config}.

\begin{theorem} 
\label{theorem_SummatoryFuncsOfDirCvls} 
Let $r,h: \mathbb{Z}^{+} \rightarrow \mathbb{C}$ be any 
arithmetic functions such that $r(1) \neq 0$. 
Suppose that $R(x) := \sum_{n \leq x} r(n)$, $H(x) := \sum_{n \leq x} h(n)$, and that 
$R^{-1}(x) := \sum_{n \leq x} r^{-1}(n)$ for $x \geq 1$. 
The following holds for all $x \geq 1$: 
\begin{align*} 
S_{r \ast h}(x) & \phantom{:}= \sum_{d=1}^x r(d) \times H\left(\Floor{x}{d}\right) \\ 
S_{r \ast h}(x) & \phantom{:}= \sum_{k=1}^{x} H(k) \times \left(R\left(\Floor{x}{k}\right) - 
     R\left(\Floor{x}{k+1}\right)\right). 
\end{align*} 
Moreover, for all $x \geq 1$ 
\begin{align*} 
H(x) & = \sum_{j=1}^{x} S_{r \ast h}(j) \times \left(R^{-1}\left(\Floor{x}{j}\right) - 
     R^{-1}\left(\Floor{x}{j+1}\right)\right) \\ 
     & = \sum_{k=1}^{x} r^{-1}(k) \times S_{r \ast h}(x). 
\end{align*} 
\end{theorem} 

For integers $x \geq 1$, the function $\pi(x) := \sum_{p \leq x} 1$ 
denotes the classical prime counting function 
\cite[\seqnum{A000720}]{OEIS}.
We find exact formulas for $M(x)$ by applying 
Theorem \hlocalref{theorem_SummatoryFuncsOfDirCvls} 
to the expansion of the partial sums 
(see Section \hlocalref{subSection_remark_MotivationForTheDefinitionOf_gn_v2}) 
\[
\pi(x) + 1 = \sum_{n \leq x} \sum_{d|n} \left(\omega(d) + 1\right) \mu\left(\frac{n}{d}\right), 
     \mathtext{ for } x \geq 1. 
\]

\begin{theorem} 
\label{prop_Mx_SBP_IntegralFormula} 
\begin{subequations}
For all $x \geq 1$ 
\begin{align} 
\label{prop_Mx_SBP_IntegralFormula_PartA} 
M(x) & = G(x) + \sum_{1 \leq k \leq x} |g(k)| \pi\left(\Floor{x}{k}\right) \lambda(k), \\ 
\label{prop_Mx_SBP_IntegralFormula_PartB} 
M(x) & = G(x) + 
     \sum_{1 \leq k \leq \frac{x}{2}} \left(
     \pi\left(\Floor{x}{k}\right) - \pi\left(\Floor{x}{k+1}\right) 
	\right) \times G(k), \\ 
\label{eqn_RmkInitialConnectionOfMxToGInvx_ProvedByInversion_v1} 
M(x) & = G(x) + \sum_{p \leq x} G\left(\Floor{x}{p}\right). 
\end{align} 
\end{subequations}
\end{theorem}

The new results proved in this article 
provide a new lense through which we can view $M(x)$ 
in terms of sums of auxiliary unsigned functions sign weighted by $\lambda(n)$. 

\subsubsection{Distributions of auxiliary unsigned functions}

The unsigned function $C_{\Omega}(n)$ is studied by Fr\"oberg in 
\cite{FROBERG-1968}. This function has the exact formula 
\begin{equation}
\label{eqn_proof_tag_hInvn_ExactNestedSumFormula_CombInterpetIdent_v3}
C_{\Omega}(n) = \begin{cases}
     1, & \text{if $n = 1$; } \\ 
     (\Omega(n))! \times \prod\limits_{p^{\alpha}||n} \frac{1}{\alpha!}, & \text{if $n \geq 2$. }
     \end{cases}
\end{equation} 

\begin{prop} 
\label{lemma_AbsValueOf_gInvn_FornSquareFree_v1} 
For all $n \geq 1$ 
\begin{equation} 
\label{eqn_AbsValueOf_gInvn_FornSquareFree_v1} 
|g(n)| = \sum_{d|n} \mu^2\left(\frac{n}{d}\right) C_{\Omega}(d). 
\end{equation} 
\end{prop} 

\begin{theorem} 
\label{lemma_HatCAstxSum_ExactFormulaWithError_v1} 
As $n \rightarrow \infty$ 
\[
\frac{1}{n} \times \sum_{k \leq n} \log C_{\Omega}(k) = 
     (\log\log n)(\log\log\log n) \left(1 + 
     O\left(\frac{1}{\sqrt{\log\log n}}\right)\right). 
\] 
\end{theorem} 

A proof of Theorem \hlocalref{lemma_HatCAstxSum_ExactFormulaWithError_v1} is 
given in Appendix \hlocalref{Appendix_ProofOfCOmegan_LogarithmicAvgOrderFormula} 
(\cf Proposition \hlocalref{prop_VarianceStat_for_COmegan_v1}). 

\begin{definition}
The cumulative density function of the 
standard normal distribution at $z$ is 
\[
\Phi(z) = \frac{1}{\sqrt{2\pi}} \times \int_{-\infty}^{z} e^{-\frac{t^2}{2}} dt, 
     \mathtext{ for any } z \in (-\infty, \infty). 
\]
\end{definition}

\begin{theorem}
\label{conj_DetFormOfEKTypeThmForCOmegan_v1} 
For $x \geq 19$, let $\mu_x, \sigma_x := (\log\log x)(\log\log\log x)$.
For any $z \in (-\infty, \infty)$ 
\begin{equation} 
\notag
\lim_{x \rightarrow \infty} \ \frac{1}{x} \times 
	\#\left\{19 \leq n \leq x: \frac{\log C_{\Omega}(n) - \mu_x}{\sigma_x} \leq z\right\} = 
     \Phi\left(z\right). 
\end{equation}
\end{theorem} 

\section{The function $C_{\Omega}(n)$} 
\label{Section_NewFormulasForgInvn_v1} 

In this section, we define the function 
$C_{\Omega}(n)$ and explore its properties. 
The function $C_{\Omega}(n)$ is key to understanding the 
unsigned inverse sequence $|g(n)|$ through equation 
\eqref{eqn_AbsValueOf_gInvn_FornSquareFree_v1}. 

\subsection{Definitions}

\begin{definition}
We define the following bivariate sequence for integers $n \geq 1$ and $k \geq 0$: 
\begin{align} 
\label{eqn_CknFuncDef_v2} 
C_k(n) := \begin{cases} 
     \varepsilon(n), & \text{ if $k = 0$; } \\ 
     \sum\limits_{d|n} \omega(d) C_{k-1}\left(\frac{n}{d}\right), & \text{ if $k \geq 1$. } 
     \end{cases} 
\end{align} 
Using the notation for iterated convolution in 
Bateman and Diamond \cite[Def.~2.3; \S 2]{ANT-BATEMAN-DIAMOND}, we have 
$C_0(n) \equiv \omega^{\ast 0}(n)$ and $C_k(n) \equiv \omega^{\ast k}(n)$ for 
integers $n, k \geq 1$. 
The special case of \eqref{eqn_CknFuncDef_v2} where 
$k := \Omega(n)$ occurs frequently in the next sections of the 
article. To avoid cumbersome notation when referring to this common variant, we suppress the 
duplicate index $n$ by writing $C_{\Omega}(n) := C_{\Omega(n)}(n)$ \cite[\seqnum{A008480}]{OEIS}. 
\end{definition}

\begin{remark}
By recursively expanding the definition of $C_k(n)$ 
at any fixed $n \geq 2$, we see that 
we can form a chain of at most $\Omega(n)$ iterated (or nested) divisor sums by 
unfolding the definition of \eqref{eqn_CknFuncDef_v2} inductively. 
We also see that at fixed $n$, the function 
$C_k(n)$ is non-zero only possibly when 
$1 \leq k \leq \Omega(n)$ when $n \geq 2$. 
By equation \eqref{eqn_proof_tag_hInvn_ExactNestedSumFormula_CombInterpetIdent_v3} we have 
that $C_{\Omega}(n) \leq (\Omega(n))!$ for all $n \geq 1$ with 
equality precisely at the squarefree integers. 
\end{remark}

\subsection{Logarithmic variance}
\label{subSection_AvgOrdersOfTheUnsignedSequences} 

\begin{definition}
\label{def_AvgOrder_FirstAndSecondMomentsOfFuncs_v1}
For any integers $x \geq 1$, we define the expectation (or mean value) of the 
function $\log C_{\Omega}(n)$ on the integers $1 \leq n \leq x$ by 
\[
\mathbb{E}\left[\log C_{\Omega}(x)\right] := \frac{1}{x} \times \sum_{n \leq x} 
     \log C_{\Omega}(n). 
\]
The variance of this function is given by the 
centralized second moments 
\[
\operatorname{Var}\left[\log C_{\Omega}(x)\right] := 
	\frac{1}{x} \times \sum_{n \leq x} \left(\log C_{\Omega}(n) - 
	\mathbb{E}\left[\log C_{\Omega}(x)\right]\right)^2. 
\]
\end{definition}

\begin{prop}
\label{prop_VarianceStat_for_COmegan_v1}
\label{prop_COmeganFunc_Variance_v1}
For all sufficiently large $n > e^e$ 
\[
\sqrt{\operatorname{Var}\left[\log C_{\Omega}(n)\right]} = 
     (\log\log n) (\log\log\log n) \left(1 + 
     O\left(\frac{1}{\sqrt{\log\log n}}\right)\right), 
	\mathtext{ as } n \rightarrow \infty. 
\]
\end{prop}
\begin{proof}
We have that for all $n \geq 1$  
\begin{equation}
\label{eqn_proof_tag_S2OmeganSum_v1}
S_{2,\Omega}(n) := \left(\sum_{k \leq n} \log C_{\Omega}(k)\right)^2 - 
     \sum_{k \leq n} \log^2 C_{\Omega}(k) = 
     \sum_{1 \leq j < k \leq n} 2 \log C_{\Omega}(j) \log C_{\Omega}(k).
\end{equation}
We define the sums 
\begin{align*}
E_{\Omega}(n) & := \frac{1}{n} \times \sum_{k \leq n} \log C_{\Omega}(k), 
     \mathtext{ and } V_{\Omega}(n) := \sqrt{\frac{1}{n} \times \sum_{k \leq n} \log^2 C_{\Omega}(k)}, 
     \mathtext{ for } n \geq 1.
\end{align*}
\begin{subequations}
\label{eqn_proof_tag_COmeganVariance_InitEqn_v1}
We have that 
\begin{align}
\label{eqn_proof_tag_COmeganVariance_InitEqn_v1_v1}
S_{2,\Omega}(n) & = n^2 E_{\Omega}^2(n) - n V_{\Omega}^2(n). 
\end{align}
The expansion on the right-hand-side of \eqref{eqn_proof_tag_S2OmeganSum_v1} is rewritten as 
\begin{align}
\label{eqn_proof_tag_COmeganVariance_InitEqn_v1_v2}
S_{2,\Omega}(n) & = 
     \sum_{1 \leq j < n} 2 \log C_{\Omega}(j) \left(
     n E_{\Omega}(n) - j E_{\Omega}(j)\right), \\ 
\notag 
     & = 2n^2 E_{\Omega}^2(n) - 2 \times \int_{e^e}^{n-1} t E_{\Omega}(t) \times \frac{d}{dt}\left[t E_{\Omega}(t)\right] dt, \\ 
\notag
     & = n^2 E_{\Omega}^2(n) - (2n-1) E_{\Omega}^2(n) \left(1 + O\left(\frac{1}{n}\right)\right). 
\end{align}
\end{subequations}
Equations \eqref{eqn_proof_tag_COmeganVariance_InitEqn_v1_v1} and 
\eqref{eqn_proof_tag_COmeganVariance_InitEqn_v1_v2} show that 
\begin{align}
\label{eqn_proof_tag_COmeganVariance_InitEqn_v2}
V_{\Omega}^2(n) &= 2 E_{\Omega}^2(n) \left(1 + O\left(\frac{1}{n}\right)\right), 
     \mathtext{ as } n \rightarrow \infty. 
\end{align} 
The variance of the function $\log C_{\Omega}(n)$ is given by
$\sigma_n^2 = V_{\Omega}^2(n) - E_{\Omega}^2(n)$.
Theorem \hlocalref{lemma_HatCAstxSum_ExactFormulaWithError_v1} applied to the right-hand-side of 
equation \eqref{eqn_proof_tag_COmeganVariance_InitEqn_v2} completes the proof. 
\end{proof}

\subsection{Remarks on other methods} 
\label{subSection_RemarksOnAvgOrderFor_COmegan_directly_SelbergDelangeMethod} 

Asymptotic formulae for the moments of the function $C_{\Omega}(n)$ on the 
positive integers $n \leq x$ as $x \rightarrow \infty$ are required 
to evaluate the average order of $|g(n)|$.  
Proofs to evaluate the centralized moments of the former function are not 
nearly as straightforward as the methods we used to establish 
Theorem \hlocalref{lemma_HatCAstxSum_ExactFormulaWithError_v1} and 
Proposition \hlocalref{prop_VarianceStat_for_COmegan_v1}. 
Let the parameters $k \in \mathbb{Z}^{+}$ and $z \in \mathbb{C}$ subject to 
$1 \leq k \leq R \log\log x$ and $|z| \leq M$ for some bounded
$0 < R, M < +\infty$ be fixed. 
An approach to the average order of $C_{\Omega}(n)$ 
invokes the Selberg-Delange method 
\cite[\S II.6.1]{TENENBAUM-PROBNUMT-METHODS} \cite[\S 7.4]{MV} 
in evaluating the partial sums of the form 
\[
\sum_{\substack{n \leq x \\ \Omega(n)=k}} \frac{(-1)^{\omega(n)} 
     C_{\Omega}(n) z^{2\Omega(n)}}{(\Omega(n))!}; \mathtext{ and } 
     \sum_{\substack{n \leq x \\ \Omega(n)=k}} \frac{(-1)^{\omega(n)} 
     C_{\Omega}(n)}{(\Omega(n))!} 
\]
We can extract the coefficients of $z^{2\Omega(n)}$ 
from the DGF expansions  
\[
\sum_{n \geq 1} \frac{C_{\Omega}(n)}{(\Omega(n))!} \cdot 
     \frac{(-1)^{\omega(n)} z^{\Omega(n)}}{n^s} = \prod_p \left(1 + \sum_{r \geq 1} 
     \frac{z^{\Omega(p^r)}}{r! p^{rs}}\right)^{-1} 
     = \exp\left(-z P(s)\right), \mathtext{ for } \Re(s) > 1. 
\]
A proof of the average order of the 
function $(-1)^{\omega(n)} C_{\Omega}(n)$ requires technical arguments 
that are non-trivial extensions of the proofs in 
\cite{MV,TENENBAUM-PROBNUMT-METHODS}. 
Integration by parts and the mean value theorem applied to the 
signed sums in Lemma \hlocalref{lemma_ConvenientIncGammaFuncTypePartialSumAsymptotics_v2} 
yield an approach to evaluating asymptotic formulae for the restricted partial sums of 
the function $C_{\Omega}(n)$ over the $n \leq x$ such that $\Omega(n) = k$ 
when $1 \leq k \leq R \log\log x$. 

\section{The function $g(n)$} 
\label{Section_NewFormulasForgInvn_v2} 

\subsection{Definitions}
\label{subSection_remark_MotivationForTheDefinitionOf_gn_v2}

\begin{definition}
\label{def_gn_and_Absgn_v2} 
For integers $n \geq 1$, we define the inverse function 
with respect to the operation of Dirichlet convolution by 
\[
g(n) = (\omega + \mathds{1})^{-1}(n), \mathtext{ for } n \geq 1. 
\]
The function $|g(n)|$ denotes the unsigned inverse function, 
or equivalently the absolute value of $g(n)$, for all integers $n \geq 1$. 
\end{definition}

\begin{remark}[Motivation]
Let $\chi_{\mathbb{P}}(n)$ denote the characteristic function of the primes, suppose that 
$\varepsilon(n) = \delta_{n,1}$ is the multiplicative identity 
with respect to Dirichlet convolution, and define the 
function $\mathds{1}(n)$ to be identically equal to one for all $n \geq 1$. 
We find that 
\begin{equation}
\label{eqn_AntiqueDivisorSumIdent} 
\chi_{\mathbb{P}} + \varepsilon = (\omega + \mathds{1}) \ast \mu. 
\end{equation} 
The result in \eqref{eqn_AntiqueDivisorSumIdent} follows by M\"obius inversion 
since $\mu \ast \mathds{1} = \varepsilon$ and 
\[
\omega(n) = \sum_{d|n} \chi_{\mathbb{P}}(d), \mathtext{ for } n \geq 1. 
\]
Recall the following statement of the 
inversion theorem of summatory functions for any 
Dirichlet invertible arithmetic function $\alpha(n)$ 
proved in \cite[\S 2.14]{APOSTOLANUMT}:
\begin{equation}
\label{eqn_ApostolStmt_ClassicSummatoryFuncInvThm_v1} 
G(x) = \sum_{n \leq x} \alpha(n) F\left(\frac{x}{n}\right) \implies 
     F(x) = \sum_{n \leq x} \alpha^{-1}(n) G\left(\frac{x}{n}\right), 
     \mathtext{ for } x \geq 1. 
\end{equation}
Hence, we may consider the inversion of the following partial sums to study the Mertens function: 
\[
\pi(x) + 1 = \sum_{n \leq x} \left(\chi_{\mathbb{P}} + \varepsilon\right)(n) = 
     \sum_{n \leq x} \sum_{d|n} (\omega(d) + 1) \mu\left(\frac{n}{d}\right), 
	\mathtext{ for } x \geq 1. 
\]
\end{remark}

\subsection{Signedness}
\label{Section_PrelimProofs_Config} 
\label{subSection_ProofOfSignednessOfgInvn_v1} 

\begin{prop}
\label{prop_SignageDirInvsOfPosBddArithmeticFuncs_v1} 
The sign of the function $g(n)$ is $\lambda(n)$ for all $n \geq 1$. 
\end{prop} 
\begin{proof} 
The series $D_f(s) := \sum_{n \geq 1} f(n) n^{-s}$ defines the 
Dirichlet generating function (DGF) of any 
arithmetic function $f$ which is convergent for all $s \in \mathbb{C}$ satisfying 
$\Re(s) > \sigma_f$ where $\sigma_f$ is the abscissa of convergence of the series. 
Recall that $D_{\mathds{1}}(s) = \zeta(s)$, $D_{\mu}(s) = \zeta(s)^{-1}$ and 
$D_{\omega}(s) = P(s) \zeta(s)$ for $\Re(s) > 1$. 
Whenever $f(1) \neq 0$ the DGF of $f^{-1}(n)$ is $D_f(s)^{-1}$. 
By equation \eqref{eqn_AntiqueDivisorSumIdent} we have 
\begin{align} 
\label{eqn_DGF_of_gInvn} 
D_{(\omega+1)^{-1}}(s) = \frac{1}{\zeta(s) (1+P(s))}, \mathtext{ for } \Re(s) > 1. 
\end{align} 
It follows that $(\omega + \mathds{1})^{-1}(n) = (h^{-1} \ast \mu)(n)$ for 
$h := \chi_{\mathbb{P}} + \varepsilon$. 
We first show that $\operatorname{sgn}(h^{-1}) = \lambda$ which implies that 
$\operatorname{sgn}(h^{-1} \ast \mu) = \lambda$. 

We recover exactly that \cite[\cf \S 2]{FROBERG-1968} 
\begin{equation} 
\notag 
h^{-1}(n) = \begin{cases} 
     1, & \text{ if $n = 1$; } \\ 
     \lambda(n) (\Omega(n))! \times \prod\limits_{p^{\alpha} || n} \frac{1}{\alpha!}, & 
     \text{ if $n \geq 2$. }
     \end{cases}
\end{equation} 
In particular, by expanding the DGF of 
$h^{-1}$ formally in powers of $P(s)$, where $|P(s)| < 1$ whenever $\Re(s) > 1.39944$, 
we count that 
\begin{align}
\notag
\frac{1}{1+P(s)} & = \sum_{n \geq 1} \frac{h^{-1}(n)}{n^s} = \sum_{k \geq 0} (-1)^k P(s)^k \\ 
\notag
     & = 
     1 + \sum_{\substack{n \geq 2 \\ n =p_1^{\alpha_1} \times \cdots \times p_k^{\alpha_k}}} 
     \frac{(-1)^{\alpha_1+\alpha_2+\cdots+\alpha_k}}{n^s} \times 
     \binom{\alpha_1+\alpha_2+\cdots+\alpha_k}{\alpha_1,\alpha_2,\ldots,\alpha_k} \\ 
\label{eqn_COmeganMultinomExp_as_DGFSeries_v2}
     & = 
     1 + \sum_{\substack{n \geq 2 \\ n =p_1^{\alpha_1} \times \cdots \times p_k^{\alpha_k}}} 
     \frac{\lambda(n)}{n^s} \times \binom{\Omega(n)}{\alpha_1,\alpha_2,\ldots,\alpha_k}. 
\end{align}
Since $\lambda$ is completely multiplicative we have that 
$\lambda\left(\frac{n}{d}\right) \lambda(d) = \lambda(n)$ for all divisors 
$d|n$ and $n \geq 1$. We also have that $\mu(n) = \lambda(n)$ 
whenever $n$ is squarefree. This yields 
\[
g(n) = (h^{-1} \ast \mu)(n) = \lambda(n) \times 
     \sum_{d|n} \mu^2\left(\frac{n}{d}\right) |h^{-1}(n)|, \mathtext{ for } n \geq 1. 
     \qedhere 
\]
\end{proof} 

\begin{remark}
The function $|h^{-1}(n)|$ from the last proof identically matches  
values of $C_{\Omega}(n)$ at all $n \geq 1$. 
The proof shows that the sequence $\lambda(n) C_{\Omega}(n)$ has 
DGF of $(1 + P(s))^{-1}$ for all $\Re(s) > 1$. We can easily extend the last proof to 
see that $C_{\Omega}(n)$ has DGF $(1-P(s))^{-1}$ for all $\Re(s) > 1$ 
(see Remark \hlocalref{remark_MiscConsequencesOfCorForFormulaOfUnsgInvnFunc_v2}). 
\end{remark}

\subsection{Relations to the function $C_{\Omega}(n)$} 
\label{Section_InvFunc_PreciseExpsAndAsymptotics} 
\label{subSection_Relating_CknFuncs_to_gInvn} 

\begin{lemma} 
\label{lemma_AnExactFormulaFor_gInvByMobiusInv_v1} 
For all $n \geq 1$ 
\[
g(n) = \sum_{d|n} \mu\left(\frac{n}{d}\right) \lambda(d) C_{\Omega}(d). 
\]
\end{lemma}
\begin{proof} 
We expand the recurrence relation for the Dirichlet inverse 
with $g(1) = 1$ as 
\begin{align} 
\label{eqn_proof_tag_gInvCvlOne_EQ_omegaCvlgInvCvl_v1} 
g(n) & = - \sum_{\substack{d|n \\ d>1}} (\omega(d) + 1) g\left(\frac{n}{d}\right) 
     \quad\implies\quad 
     (g \ast 1)(n) = -(\omega \ast g)(n). 
\end{align} 
For $1 \leq m \leq \Omega(n)$, we can inductively expand the 
implication on the right-hand-side of \eqref{eqn_proof_tag_gInvCvlOne_EQ_omegaCvlgInvCvl_v1} 
in the form of $(g \ast 1)(n) = G_m(n)$ where 
$$G_m(n) := (-1)^m \times \sum_{d|n} C_m(d) g\left(\frac{n}{d}\right), \mathtext{ for } 1 \leq m \leq \Omega(n),$$ 
is expanded recursively as 
\[
G_m(n) = - 
     \begin{cases} 
     (\omega \ast g)(n), & m = 1; \\ 
     \sum\limits_{\substack{d|n \\ d > 1}} G_{m-1}(d) \times \sum\limits_{\substack{r|\frac{n}{d} \\ r > 1}} 
     \omega(r) g\left(\frac{n}{dr}\right), & 2 \leq m \leq \Omega(n); \\ 
     0, & \text{otherwise.} 
     \end{cases} 
\]
When $n \geq 2$ and $m := \Omega(n)$, i.e., with these expansions 
carried out to a maximal depth, we obtain 
\begin{equation} 
\label{eqn_proof_tag_gInvCvlOne_EQ_omegaCvlgInvCvl_v2} 
(g \ast 1)(n) = \lambda(n) C_{\Omega}(n). 
\end{equation} 
The formula follows from equation \eqref{eqn_proof_tag_gInvCvlOne_EQ_omegaCvlgInvCvl_v2} 
by M\"obius inversion. 
\end{proof} 

\begin{proof}[Proof of Proposition \hlocalref{lemma_AbsValueOf_gInvn_FornSquareFree_v1}] 
The result follows from 
Lemma \hlocalref{lemma_AnExactFormulaFor_gInvByMobiusInv_v1}, 
Proposition \hlocalref{prop_SignageDirInvsOfPosBddArithmeticFuncs_v1} and the 
complete multiplicativity of $\lambda(n)$.  
Since $\mu(n)$ is non-zero only at squarefree integers and since 
at any squarefree $d \geq 1$ we have $\mu(d) = (-1)^{\omega(d)} = \lambda(d)$, we have 
\begin{align*} 
|g(n)| & = \lambda(n) \times \sum_{d|n} \mu\left(\frac{n}{d}\right) \lambda(d) C_{\Omega}(d) \\ 
     & = \lambda(n^2) \times \sum_{d|n} \mu^2\left(\frac{n}{d}\right) C_{\Omega}(d). 
\end{align*} 
The leading term $\lambda(n^2) = 1$ for all $n \geq 1$ since the number of distinct 
prime factors (counting multiplicity) of any square integer is even. 
\end{proof} 

\begin{remark}
\label{remark_MiscConsequencesOfCorForFormulaOfUnsgInvnFunc_v2} 
The following are consequences of 
Proposition \hlocalref{lemma_AbsValueOf_gInvn_FornSquareFree_v1}: 
\begin{subequations}
\begin{itemize}
\item 
Whenever $n \geq 1$ is squarefree 
\begin{equation}
|g(n)| = \sum_{d|n} C_{\Omega}(d). 
\end{equation}
Since all divisors of a squarefree integer are squarefree, 
for all squarefree $n \geq 1$, we have that 
\begin{equation}
\label{eqn_PropB_lemma_gInv_MxExample} 
|g(n)| = \sum_{m=0}^{\omega(n)} \binom{\omega(n)}{m} \times m!. 
\end{equation}
\item 
The formula in \eqref{eqn_AbsValueOf_gInvn_FornSquareFree_v1} shows that 
the DGF of the unsigned inverse function $|g(n)|$ 
is given by the meromorphic function 
$\zeta(2s)^{-1} (1-P(s))^{-1}$ for all $s \in \mathbb{C}$ with $\Re(s) > 1$. 
This DGF has a pole to the right of the line at $\Re(s) = 1$ 
at the unique real $\sigma \approx 1.39943$ 
such that $P(\sigma) = 1$ along the reals $\sigma > 1$. 
\item 
The average order of $|g(n)|$ is given by 
\cite[\S 18.6]{HARDYWRIGHT}
\begin{align}
\frac{1}{x} \times \sum_{n \leq x}  |g(n)| & = \frac{1}{x} \times \sum_{n \leq x} 
     C_{\Omega}(n) \left(\sum_{k \leq \Floor{x}{n}} \mu^2(k)\right) 
     = \frac{6}{\pi^2} \times \sum_{n \leq x} \frac{C_{\Omega}(n)}{n}\left( 
     1 + O\left(\sqrt{\frac{n}{x}}\right)\right). 
\end{align}
\end{itemize}
\end{subequations}
\end{remark}

\section{The distribution of the function $\log C_{\Omega}(n)$} 
\label{subSection_ErdosKacTheorem_Analogs} 

In this section, we motivate and prove a central limit theorem 
for the distribution of the function $\log C_{\Omega}(n)$. 

\subsection{Proof of Theorem \hlocalref{conj_DetFormOfEKTypeThmForCOmegan_v1}}

\begin{proof}
We outline the next steps to complete the proof of this result: 
\begin{itemize}
\item
Given a fixed $x \geq 1$, we select another integer $N \equiv N(x)$ uniformly at random from 
$\{1,2,\ldots,x\}$. For each prime $p$ we define 
\[
C_p^{(x)} := \begin{cases} 0, & p \nmid N(x); \\ 
	\alpha, & p^{\alpha} || N(x), \text{ for some } \alpha \geq 1. 
	\end{cases}
\]
For primes $p$ as $x \rightarrow \infty$, 
we have the limiting convergence in distribution of 
$C_p^{(x)} \overset{d}{\implies} Z_p$ where $Z_p$ is 
a geometric random variable with parameter $p^{-1}$ 
\citep[\S 1.2]{LOG-COMB-STRUCTS-BOOK}. 
In other words, for any prime $p$ and $k \geq 1$ we have that 
\[
\lim_{x \rightarrow \infty} \mathbb{P}\left[C_p^{(x)} = k\right] = 
     \left(1 - \frac{1}{p}\right)\left(\frac{1}{p}\right)^k. 
\]
\item 
For $n \geq 1$, we use 
equation \eqref{eqn_proof_tag_hInvn_ExactNestedSumFormula_CombInterpetIdent_v3} and 
Binet's log-gamma formula \cite[\S 5.9(i)]{NISTHB} to show that 
\begin{align}
\notag
\log C_{\Omega}(n) & = \log (\Omega(n))! - 
	\sum_{\substack{p^{\alpha}||n \\ \alpha \geq 2}} \log(\alpha!) \\ 
\label{eqn_remark_LogCOmegan_IntermedExp_v1}
	& = \Omega(n) \log \Omega(n) - \sum_{\substack{p^{\alpha}||n \\ \alpha \geq 2}} 
	\alpha \log(1 + \alpha) + O(\Omega(n)). 
\end{align}
Since $\Omega(n) = 1$ only for $n$ within a subset of the positive 
integers with asymptotic density of zero (i.e., at prime $n$), 
it suffices to restrict our considerations 
to the cases where $\Omega(n) \geq 2$. 
\item 
For $x \geq 2$, let 
\begin{align*}
\Theta_{N(x)} & := \Omega(N(x)) \log \Omega(N(x)), \\ 
A_{N(x)} & := \sum_{p \leq x} C_p^{(x)} \log C_p^{(x)} \times \mathds{1}_{\{C_p^{(x)} \geq 2\}}(p). 
\end{align*}
We can write the expansion from equation 
\eqref{eqn_remark_LogCOmegan_IntermedExp_v1} as the difference 
$$\log C_{\Omega}(N(x)) := \Theta_{N(x)} - A_{N(x)} + O(1), \mathtext{ as } x \rightarrow \infty.$$ 
Moreover, we can show that as $x \rightarrow \infty$ 
\[
\mathbb{E}[A_{N(x)}] \ll \sum_{p \leq x} \mathbb{E}\left[C_p^{(x)} \log C_p^{(x)}\right] 
     \times \mathbb{P}\left[C_p^{(x)} \geq 2\right] = o\left(\mathbb{E}[\Theta_{N(x)}]\right). 
\]
Analogous bounds can be proved to relate the variance of these 
two random variables as $x \rightarrow \infty$. 
\item 
Let $\mu_x := \mathbb{E}\left[\log C_{\Omega}(x)\right]$ and 
$\sigma_x^2 := \operatorname{Var}\left[\log C_{\Omega}(x)\right]$ be defined as in 
Definition \hlocalref{def_AvgOrder_FirstAndSecondMomentsOfFuncs_v1}. 
For $1 \leq n \leq x$, let the indicator random variable $\chi_{x,n}$ 
be defined as follows: $\chi_{x,n} := \mathds{1}_{\{N(x)=n\}}$. 
For $x \geq 1$, we define 
\[
S_x := \sum_{1 \leq n \leq x} \log C_{\Omega}(n) \chi_{x,n}. 
\]
We calculate that 
\[
\mathbb{E}[S_x] = \mu_x; \operatorname{Var}[S_x] = \sigma_x^2; 
     \mathtext{ and } 
     \hat{\mu}_{x,n} := \mathbb{E}\left[\log C_{\Omega}(n) \chi_{x,n}\right] = 
     \frac{1}{x} \times \log C_{\Omega}(n), 
     \mathtext{ for integers } 1 \leq n \leq x. 
\]
\item 
For fixed $\varepsilon > 0$ and large $x$, let 
\[
\widetilde{E}_{\Omega}(\varepsilon, x) := \frac{1}{\sigma_x^2} \times 
     \sum_{1 \leq n \leq x} \mathbb{E}\left[ 
     \left(\log C_{\Omega}(n) \chi_{x,n} - \hat{\mu}_{x,n}\right)^2 
     \times 
     \mathds{1}_{\{\left\lvert \log C_{\Omega}(n) \chi_{x,n} - \hat{\mu}_{x,n} \right\rvert > 
     \varepsilon \sigma_x\}}\right]. 
\]
We say that the Lindeberg condition is satisfied when the following 
is true for every fixed $\varepsilon > 0$: 
\begin{equation}
\label{eqn_ProofTag_LindebergCLTCond_Stmt_v1}
\lim_{x \rightarrow \infty}\ \widetilde{E}_{\Omega}(\varepsilon, x) = 0. 
\end{equation}
Whenever equation \eqref{eqn_ProofTag_LindebergCLTCond_Stmt_v1} holds for all $\varepsilon > 0$, 
we can apply the Lindeberg central limit theorem (CLT). This result and 
Theorem \hlocalref{lemma_HatCAstxSum_ExactFormulaWithError_v1} and 
Proposition \hlocalref{prop_COmeganFunc_Variance_v1} show that we have the 
convergence in distribution to a standard normal random variable given as follows 
\cite[\S 27]{BILLINGSLY-PROB-AND-MEASURE-BOOK}: 
\begin{equation}
\label{eqn_ProofTag_LindebergCLT_Conclusion_Stmt_v2}
\lim_{x \rightarrow \infty} 
     \mathbb{P}\left[\frac{S_x - \mu_x}{\sigma_x} \leq z\right] = 
	\Phi\left(z\right), \mathtext{ for any } z \in (-\infty, \infty).
\end{equation}
\item 
For $x,y \in [0, \infty)$, the function $W(y)$ 
denotes the principal branch of the multi-valued 
Lambert $W$-function on the non-negative reals defined such that 
$x = W(y)$ if and only if $xe^{x} = y$.
For any $M > 0$, the condition that $k\log k > M \sigma_x$ is true when 
$k > \frac{M \sigma_x}{W\left(M \sigma_x\right)} \sim M \log\log x$ 
as $x \rightarrow \infty$ \cite{LAMBERT-WFUNC-KNUTH}. 
The inequality $t \log(1+t) \geq \left(t + \frac{1}{2}\right) \log(1+t) - t$ 
is satisfied for all real $t > 0$. 
\item 
Suppose that $\varepsilon \in (0, 1)$ is fixed and 
let $\mathcal{I}(\varepsilon, x) := \left[(1-\varepsilon) \log\log x, (1+\varepsilon) \log\log x \right] \bigcap \mathbb{Z}^{+}$. 
We have that 
\begin{align} 
\notag
\widetilde{E}_{\Omega}(\varepsilon, x) & \ll 
     \frac{1}{\sigma_x^2} \times \sum_{n \leq x} \mathbb{E}\left[
     \log^2 C_{\Omega}(n) \left(\chi_{x,n} - \frac{1}{x}\right)^2 
     \mathds{1}_{\left\{\left\lvert S_x - \mu_x \right\rvert > \varepsilon\sigma_x\right\}}
     \right] \\ 
\notag
     & \ll \frac{1}{\sigma_x^2} \times \sum_{k \in \mathcal{I}(\varepsilon, x)} 
     \log^2(k!) \times \mathbb{P}\left[\Omega(N(x)) = k\right] \times \sum_{n \leq x} 
     \mathbb{E}\left(\chi_{x,n} - \frac{1}{x}\right)^2 \\ 
\label{eqn_ProofTag_CLT_VerifyLindebergCondition_v1_v3}
     & \ll \frac{1}{x \sigma_x^2} \times \sum_{k \in \mathcal{I}(\varepsilon, x)} 
     \log^2(k!) \times \mathbb{P}\left[\Omega(N(x)) = k\right] \times \left(1 - \frac{1}{x}\right). 
\end{align}
Proposition \hlocalref{prop_VarianceStat_for_COmegan_v1}, 
Theorem \hlocalref{theorem_MV_Thm7.20-init_stmt} and 
equation \eqref{eqn_ProofTag_CLT_VerifyLindebergCondition_v1_v3} show that 
\begin{align*}
\widetilde{E}_{\Omega}(\varepsilon, x) & \ll 
      \left\lvert (\log x)^{-\varepsilon-(1-\varepsilon) \log(1-\varepsilon)} - 
      (\log x)^{\varepsilon-(1+\varepsilon) \log(1+\varepsilon)} \right\rvert. 
\end{align*}
The function $f_{\pm}(t) := \pm t - (1 \pm t) \log(1 \pm t)$ is strictly negative and 
monotone decreasing for $t \in (0, 1)$. 
It follows that equation \eqref{eqn_ProofTag_LindebergCLTCond_Stmt_v1} 
is satisfied for all $\varepsilon \in (0, 1)$. 
We can modify the argument given above using equation \eqref{eqn_proof_tag_PartialSumsOver_HatCkx_EquivCond_v2} 
from the appendix to show that 
\eqref{eqn_ProofTag_LindebergCLTCond_Stmt_v1} is also satisfied 
whenever $\varepsilon \in [1, \infty)$. 
We conclude from the Lindeberg CLT that 
\eqref{eqn_ProofTag_LindebergCLT_Conclusion_Stmt_v2} holds. 
\qedhere
\end{itemize}
\end{proof}

\subsection{Motivation}

\begin{remark}
\label{remark_COmegaFuncDistIntutitionFromErdosKac_v1}
For $n \geq 2$, let the function 
$\mathcal{E}[n] := (\alpha_1, \ldots, \alpha_r)$ denote the unordered 
partition of exponents ($r$-tuple) for which $\omega(n) = r$ and 
$n = p_1^{\alpha_1} \times \cdots \times p_r^{\alpha_r}$ is the factorization of 
$n$ into powers of distinct primes. 
For any $n_1,n_2 \geq 2$ 
\begin{equation}
\label{eqn_FactSymmPropertyOfgn_v1} 
\mathcal{E}[n_1] = \mathcal{E}[n_2] \implies C_{\Omega}(n_1) = C_{\Omega}(n_2) \mathtext{ and } 
	g(n_1) = g(n_2). 
\end{equation}
This property shows that there is a deep structure to these functions connected to the 
prime divisors of the positive integers $n \geq 2$. 
On the other hand, since the multiplicative function $\mu^2(n)$ and the 
strongly additive functions $\omega(n)$ and $\Omega(n)$ also satisfy their 
analog to equation \eqref{eqn_FactSymmPropertyOfgn_v1}. Thus, 
this property alone is insufficient to predict the CLT type 
result for functions of this type we see in 
Theorem \hlocalref{conj_DetFormOfEKTypeThmForCOmegan_v1}. 
\end{remark}

More intuition about why the distribution of $\log C_{\Omega}(n)$ 
obeys a limiting probabilistic model is heuristic: 

\begin{remark}
By definition the function $C_{\Omega}(n)$ is identified with the $\Omega(n)$-fold 
Dirichlet convolution of the strongly additive $\omega(n)$ with itself via 
Definition \hlocalref{eqn_CknFuncDef_v2}. 
This perspective provides more insight into why we should expect to find a limiting distribution 
associated with the distinct values of 
$C_{\Omega}(n)$ over $n \leq x$ (pointwise) and of $\log C_{\Omega}(n)$ over $n \leq x$ 
(smoothly via Theorem \hlocalref{conj_DetFormOfEKTypeThmForCOmegan_v1}). 
In particular, we associate the tendency of $\omega(n)$ towards its average order 
with the Erd\H{o}s-Kac theorem 
(\cf Appendix \hlocalref{subSection_TheKnownDistsOfThePrimeOmegaFunctions_IntroResults_v1}) 
\[
\lim_{x \rightarrow \infty}\ \frac{1}{x} \times \#\left\{n \leq x: 
     \frac{\omega(n) - \log\log x}{\sqrt{\log\log x}} \leq z\right\} = \Phi(z), 
	\mathtext{ for any } z \in (-\infty, \infty). 
\]
In the sense that multiple (Dirichlet) convolutions 
reflect a qualitative smoothing operation on average, the CLT statement for $\omega(n)$ above 
should predict a smooth limiting distribution. 
Incidentally, equation \eqref{eqn_proof_tag_hInvn_ExactNestedSumFormula_CombInterpetIdent_v3} 
shows that the normalized function $\frac{C_{\Omega}(n)}{(\Omega(n))!}$ is multiplicative 
(\cf \cite{ELLIOTT-V1}). 
\end{remark}

\section{Applications to the Mertens function} 
\label{Section_KeyApplications} 
\label{Section_KeyApplications_NewExactFormulasForMx_FullSectionLabel} 

In this section, we prove Theorem \hlocalref{prop_Mx_SBP_IntegralFormula}. 
The new formulas precisely connect the Mertens function with partial sums of 
positive unsigned arithmetic functions whose summands are 
weighted by the sign of $\lambda(n)$.  

\begin{definition}
\label{def_GInvAndGInvAbs_SummFuncs_v2}
\begin{subequations}
The summatory functions of $g(n)$ and $|g(n)|$, respectively, are 
defined for all $x \geq 1$ by the partial sums 
\begin{align*}
G(x) & := \sum_{n \leq x} g(n) = \sum_{n \leq x} \lambda(n) |g(n)|, 
	\mathtext{ and } 
	|G|(x) := \sum_{n \leq x} |g(n)|. 
\end{align*}
\end{subequations}
\end{definition}

\begin{figure}[ht!]

\captionsetup{singlelinecheck=off}
\centering

\begin{subfigure}[t!]{\PlotFigureHorizontalScalingFactor\textwidth}
\fbox{\includegraphics[width=\textwidth]{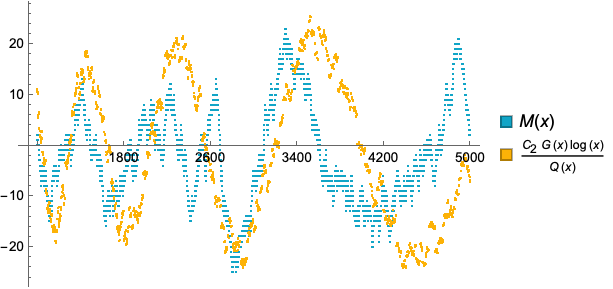}}
\captionsetup{justification=centering}
\caption{}
\end{subfigure}

\smallskip

\begin{subfigure}[t!]{\PlotFigureHorizontalScalingFactor\textwidth}
\fbox{\includegraphics[width=\textwidth]{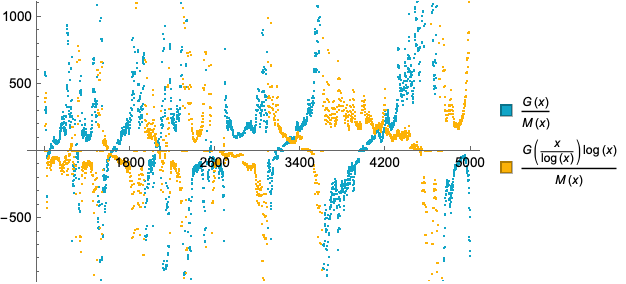}}
\captionsetup{justification=centering}
\caption{}
\end{subfigure}

\captionsetup{justification=centering}
\caption{} 
\label{figure_MxAndNewAuxPartialSums_Comparison_Intro_v2_v1} 

\end{figure} 

\begin{figure}[ht!]

\captionsetup{singlelinecheck=off}
\centering

\begin{subfigure}[t!]{\PlotFigureHorizontalScalingFactor\textwidth}
\fbox{\includegraphics[width=\textwidth]{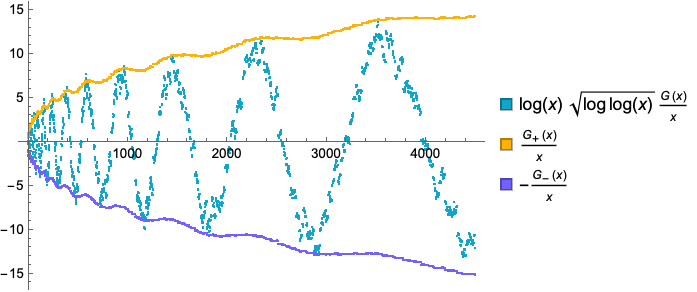}}
\captionsetup{justification=centering}
\caption{}
\end{subfigure}

\smallskip

\begin{subfigure}[t!]{\PlotFigureHorizontalScalingFactor\textwidth}
\fbox{\includegraphics[width=\textwidth]{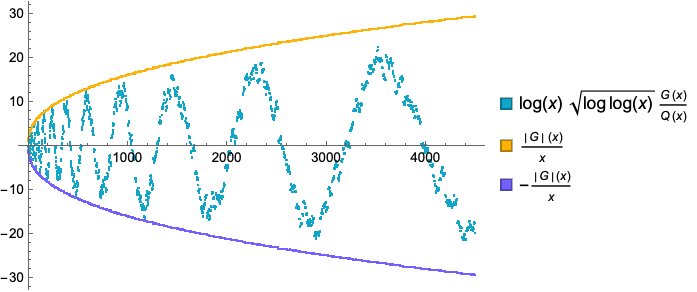}}
\captionsetup{justification=centering}
\caption{}
\end{subfigure}

\captionsetup{justification=centering}
\caption{}
\label{figure_MxAndNewAuxPartialSums_Comparison_Intro_v2_v2} 

\end{figure} 

\subsection{Proofs of the new formulas} 
\label{subSection_KeyApplications_NewExactFormulasForMx} 

\begin{proof}[Proof of 
              \eqref{prop_Mx_SBP_IntegralFormula_PartA} and \eqref{prop_Mx_SBP_IntegralFormula_PartB} of 
              Theorem \hlocalref{prop_Mx_SBP_IntegralFormula}] 
By applying Theorem \hlocalref{theorem_SummatoryFuncsOfDirCvls} to 
equation \eqref{eqn_AntiqueDivisorSumIdent} we have that 
\begin{align} 
\notag
M(x) & = \sum_{k=1}^{x} \left(1 + \pi\left(\Floor{x}{k}\right)\right) g(k) \\ 
\notag 
     & = G(x) + \sum_{k=1}^{\frac{x}{2}} \pi\left(\Floor{x}{k}\right) g(k) \\ 
\notag 
     & = G(x) + G\left(\Floor{x}{2}\right) + 
     \sum_{k=1}^{\frac{x}{2}-1} \left( 
     \pi\left(\Floor{x}{k}\right) - \pi\left(\Floor{x}{k+1}\right) 
	\right) \times G(k).
\end{align} 
The upper bound on the sum is truncated to $k \in \left[1, \frac{x}{2}\right]$ in the second equation 
above because $\pi(1) = 0$. 
The third formula above follows directly by summation by parts. 
\end{proof} 
\begin{proof}[Proof of \eqref{eqn_RmkInitialConnectionOfMxToGInvx_ProvedByInversion_v1} of 
	      Theorem \hlocalref{prop_Mx_SBP_IntegralFormula}]
Lemma \hlocalref{lemma_AnExactFormulaFor_gInvByMobiusInv_v1} shows that 
\[
G(x) = \sum_{d \leq x} \lambda(d) C_{\Omega}(d) M\left(\Floor{x}{d}\right). 
\]
The identity in \eqref{eqn_AntiqueDivisorSumIdent} implies 
$$\lambda(d) C_{\Omega}(d) = (g \ast 1)(d) = (\chi_{\mathbb{P}} + \varepsilon)^{-1}(d).$$ 
We recover the stated result from the classical inversion of summatory functions in 
equation \eqref{eqn_ApostolStmt_ClassicSummatoryFuncInvThm_v1}. 
\end{proof}

\subsection{Discrete plots and numerical experiments}

The plots shown in the figures in this section compare 
the values of $M(x)$ and $G(x)$ with related auxiliary partial sums. 
These plots showcase interesting phenomena observed for small $x$: 
\begin{itemize}

\item In Figure \hlocalref{figure_MxAndNewAuxPartialSums_Comparison_Intro_v2_v1}, 
      we plot comparisons of $M(x)$ to scaled forms of $G(x)$ for $x \leq 5000$. The 
      absolute constant is $C_2 := \frac{\pi^2}{6}$. The partial sums defined by the function 
      $Q(x) := \sum_{n \leq x} \mu^2(n)$ count the number of squarefree integers $1 \leq n \leq x$. 
      In \textbf{(a)} the shift to the left on the $x$-axis of the former function 
      is compared and seen to be similar in shape to the magnitude of $M(x)$ on this initial subinterval. 
      It is unknown whether the similar shape and magnitude of these two functions persists for 
      much larger $x$. 
      In \textbf{(b)} we have observed unusual reflections and symmetry between the two ratios plotted in the 
      figure. We have numerically modified the plot values to shift the denominators of 
      $M(x)$ by one at each $x \leq 5000$ for which $M(x) = 0$. 

\item In Figure \hlocalref{figure_MxAndNewAuxPartialSums_Comparison_Intro_v2_v2}, we compare 
      envelopes on the logarithmically scaled values of $G(x) x^{-1}$ to other variants of 
      the partial sums of $g(n)$ for $x \leq 4500$. 
      In \textbf{(a)} we define $G(x) := G_{+}(x) - G_{-}(x)$ where the functions 
      $G_{+}(x) \geq 0$ and $G_{-}(x) \geq 0$ for all $x \geq 1$, 
      i.e., the signed component functions $G_{\pm}(x)$ 
      denote the unsigned contributions of only those summands 
      $|g(n)|$ over $n \leq x$ where $\lambda(n) = \pm 1$, respectively. 
      The summatory function $Q(x)$ 
      in \textbf{(b)} has the same definition as in 
      Figure \hlocalref{figure_MxAndNewAuxPartialSums_Comparison_Intro_v2_v1} above. 
      This plot suggests that for large $x$ 
      \[
      |G(x)| \ll \frac{|G|(x)}{(\log x) \sqrt{\log\log x}} = 
           \frac{1}{(\log x) \sqrt{\log\log x}} \times \sum_{n \leq x} |g(n)|, 
           \mathtext{ as } x \rightarrow \infty.
      \]

\end{itemize}

\section{Conclusions}

\subsection{Summary}

We have identified a sequence, 
$\{g(n)\}_{n \geq 1}$, that is the Dirichlet inverse of the 
shifted strongly additive function $\omega(n)$. 
There is a natural structure of the repetition of distinct values 
of $|g(n)|$ that depends on the configuration of the 
exponents of the distinct primes in the factorization of any $n \geq 2$. 
The definition of this auxiliary sequence provides new relations between the 
summatory function $G(x)$ to $M(x)$ and $L(x)$. 
The sign of $g(n)$ is given by $\lambda(n)$ for all $n \geq 1$. 
The distributions of the unsigned functions $C_{\Omega}(n)$ and $|g(n)|$ 
provide new information about $M(x)$ via the formulas proved in 
Theorem \hlocalref{prop_Mx_SBP_IntegralFormula}. These formulas are 
expressed in terms of weighted partial sums with terms signed by $\lambda(n)$

\subsection{Discussion of the new results}

\subsubsection{Randomized models of the M\"obius function}

Some natural probabilistic models of the 
M\"obius function lead us to consider the behavior of $M(x)$ 
as a sum of independent and identically distributed (i.i.d.) random variables. 
Suppose that $\{X_n\}_{n \geq 1}$ is a sequence of i.i.d. 
$\{-1,0,1\}$-valued random variables 
such that for all $n \geq 1$ 
$$\mathbb{P}[X_n = -1] = \mathbb{P}[X_n = +1] = \frac{3}{\pi^2}, 
  \mathtext{ and } \mathbb{P}[X_n = 0] = 1 - \frac{6}{\pi^2}.$$ 
That is, the sequence $\{X_n\}_{n \geq 1}$ provides a randomized model of the 
values of $\mu(n)$ on average. 
We may approximate limiting properties of the partial sums as 
$M(x) \cong \widebar{S}_x$ where $\widebar{S}_x := \sum_{n \leq x} X_n$ for all $x \geq 1$. 
This viewpoint models predictions of certain limiting asymptotic behavior of the 
Mertens function such as \cite[Thm.~9.4; \S 9]{BILLINGSLY-PROB-AND-MEASURE-BOOK}
\[
\mathbb{E}\left[\widebar{S}_x\right] = 0, 
     \operatorname{Var}\left[\widebar{S}_x\right] = \frac{6x}{\pi^2}, 
     \mathtext{ and } 
     \limsup_{x \rightarrow \infty} \frac{\left\lvert \widebar{S}_x\right\rvert}{\sqrt{x \log\log x}} = 
     \frac{2\sqrt{3}}{\pi} \mathtext{ (almost surely).} 
\]

\subsubsection{Comparison of known formulas for $M(x)$ involving $\lambda(n)$}

The Mertens function is related to the partial sums in 
\eqref{eqn_LxSummatoryFuncDef_v1} 
via the relation \cite{HUMPHRIES-JNT-2013,LEHMAN-1960} 
\begin{equation}
\label{eqn_MxInTermsOfLx_v1} 
M(x) = \sum_{d \leq \sqrt{x}} \mu(d) L\left(\Floor{x}{d^2}\right), \mathtext{ for } x \geq 1.
\end{equation}
The relation in \eqref{eqn_MxInTermsOfLx_v1} 
gives an exact expression for $M(x)$ with summands involving $L(x)$ that are oscillatory. 
In contrast, the exact expansions for the Mertens function given in 
Theorem \hlocalref{prop_Mx_SBP_IntegralFormula} 
express $M(x)$ as finite sums over $\lambda(n)$ with weighted coefficients that are unsigned. 
The property of the symmetry of the distinct values of $|g(n)|$ with respect to the 
prime factorizations of $n \geq 2$ in equation \eqref{eqn_FactSymmPropertyOfgn_v1} 
suggests that the unsigned weights on $\lambda(n)$ in 
the new formulas from the theorem may yield new insights compared to formulas like 
equation \eqref{eqn_MxInTermsOfLx_v1}. 

\subsubsection{The unpredictability of $\lambda(n)$ versus $\mu(n)$}

Stating tight bounds on the distribution of 
$L(x)$ is a problem that is equally as difficult 
as understanding the growth of $M(x)$
along infinite subsequences (\cf \cite{MR2877066,MR3779960,TAO-LOGAVGD-CHOWLA}). 
Indeed, $\lambda(n) = \mu(n)$ for all squarefree $n \geq 1$ so that 
$\lambda(n)$ agrees with $\mu(n)$ at most large $n$. 
It can be inferred that $\lambda(n)$ must inherit the pseudo-randomized quirks 
of $\mu(n)$ predicted by Sarnak's conjecture. 
On the other hand, the formulas in 
Theorem \hlocalref{prop_Mx_SBP_IntegralFormula} are more desirable to explore than 
other classical formulae for $M(x)$ according to the rationale in the following points:
\begin{itemize}
\item Breakthrough work in recent years due to 
	 Matom\"aki, Radziwi{\l\l} and Soundararajan to 
	 bound multiplicative functions 
	 in short intervals has 
	 proven fruitful when applied to $\lambda(n)$ 
	 \cite{SOUND-LLAMBDA-SHORT-INTS,MATRADZE-MULTFUNCS-SHORT-INTS}. 
	 The analogs of results of this type corresponding 
	 to the M\"obius function are not clearly attained; 
\item The squarefree $n \geq 1$ on which $\lambda(n)$ and $\mu(n)$ must identically agree 
	 are in some senses easier integer cases to handle 
	 insomuch as we can prove limiting distributions for the distinct values of 
	 $\omega(n)$, $\Omega(n)$, and their difference, over $n \leq x$ as $x \rightarrow \infty$ 
	 \citep[\cf \S 2.4; \S 7.4]{MV}; 
\item The function $\lambda(n)$ is completely 
	 multiplicative. Hence, the values of the non-zero function $\lambda(n)$ may be 
	 more regular in certain ways on the integers $n \geq 4$ for which $\mu(n) = 0$. 
\end{itemize}

\addcontentsline{toc}{section}{Acknowledgements}
\section*{Acknowledgements}

We credit Appendix \hlocalref{subSection_OtherFactsAndResults} 
to correspondence with 
Gerg\H{o} Nemes from the Alfr\'{e}d R\'{e}nyi Institute of Mathematics. 

\renewcommand{\refname}{References} 
\addcontentsline{toc}{section}{References}
\bibliographystyle{plain}

\appendix
\cftaddtitleline{toc}{section}{Appendices on supplementary material}{}
\setcounter{section}{0} 
\renewcommand{\thesection}{\Alph{section}} 
\newpage

\section{The distributions of $\omega(n)$ and $\Omega(n)$} 
\label{subSection_TheKnownDistsOfThePrimeOmegaFunctions_IntroResults_v1} 

As $n \rightarrow \infty$, we have that 
$$\frac{1}{n} \times \sum_{k \leq n} \omega(k) \sim \log\log n + B_1,$$ 
and 
$$\frac{1}{n} \times \sum_{k \leq n} \Omega(k) \sim \log\log n + B_2,$$ 
where $B_1 \approx 0.261497$ and $B_2 \approx 1.03465$ are 
absolute constants \cite[\S 22.10]{HARDYWRIGHT}. 
The next theorems from \cite[\S 7.4]{MV} bound the frequency of the 
number of times $\Omega(n)$ $n \leq x$ 
diverges substantially from its average order at integers $n \leq x$ 
when $x$ is large 
(\cf \cite{ERDOS-KAC-REF,BILLINGSLY-CLT-PRIMEDIVFUNC}). 

\begin{theorem} 
\label{theorem_MV_Thm7.20-init_stmt} 
For $x \geq 2$ and $r > 0$, let 
\begin{align*} 
A(x, r) & := \#\left\{n \leq x: \Omega(n) \leq r \log\log x\right\}, \\ 
B(x, r) & := \#\left\{n \leq x: \Omega(n) \geq r \log\log x\right\}. 
\end{align*} 
If $0 < r \leq 1$, then 
\[
A(x, r) \ll\phantom{_R} x (\log x)^{r-1 - r\log r}, \mathtext{ as } x \rightarrow \infty. 
\]
If $1 \leq r < 2$, then 
\[
B(x, r) \ll x (\log x)^{r-1-r \log r}, \mathtext{ as } x \rightarrow \infty. 
\]
\end{theorem} 
\begin{proof}
The proof of this theorem is given in \cite[Thm.~7.20; \S 7.4]{MV}. 
It uses an adaptation of Rankin's method in combination with 
\cite[Thm.~7.18; \S 7.4]{MV} to obtain the two upper bounds. 
\end{proof}

\begin{theorem}
\label{theorem_HatPi_ExtInTermsOfGz} 
For integers $k \geq 1$ and $x \geq 2$ 
$$\widehat{\pi}_k(x) := \#\{1 \leq n \leq x: \Omega(n)=k\}.$$ 
For $0 \leq |z| < R$, we define the function 
\[
\mathcal{G}(z) := \frac{1}{\Gamma(1+z)} \times 
	\prod_p \left(1-\frac{z}{p}\right)^{-1} \left(1-\frac{1}{p}\right)^z. 
\]
For $0 < R < 2$, uniformly for $1 \leq k \leq R \log\log x$ 
\begin{equation}
\label{eqn_PiHatkx_UniformAsymptoticsStmt_from_MV_v1}
\widehat{\pi}_k(x) = \frac{x}{\log x} \times \mathcal{G}\left(\frac{k-1}{\log\log x}\right) 
     \frac{(\log\log x)^{k-1}}{(k-1)!} \times \left(1 + O\left(\frac{k}{(\log\log x)^2}\right)\right), 
     \mathtext{ as } x \rightarrow \infty. 
\end{equation}
\end{theorem}
\begin{proof}
The proof of this theorem is given in \cite[Thm.~7.19; \S 7.4]{MV}. 
The notation $\widehat{\pi}_k(x)$ is distinct from that 
other references that use $N_k(x)$ and $\sigma_k(x)$, 
respectively \citep[Eqn.~(7.61)]{MV} \citep[\cf \S II.6]{TENENBAUM-PROBNUMT-METHODS}. 
\end{proof}

\begin{theorem} 
\label{remark_MV_Pikx_FuncResultsAnnotated_v1} 
For integers $k  \geq 1$ and $x \geq 2$, we define 
\[
\pi_k(x) := \#\{2 \leq n \leq x: \omega(n)=k\}.
\]
We define the function 
\[
\widetilde{\mathcal{G}}(z) := \frac{1}{\Gamma(1+z)} \times 
	\prod_p \left(1 + \frac{z}{p-1}\right) \left(1 - \frac{1}{p}\right)^{z}, 
	\mathtext{ for } |z| \leq R < 2. 
\]
For fixed $0 < R < 2$, as $x \rightarrow \infty$ we have 
uniformly for $1 \leq k \leq R\log\log x$ that 
\begin{equation}
\label{eqn_Pikx_UniformAsymptoticsStmt_from_MV_v2} 
\pi_k(x) = \frac{x}{\log x} \times 
     \widetilde{\mathcal{G}}\left(\frac{k-1}{\log\log x}\right) 
     \frac{(\log\log x)^{k-1}}{(k-1)!} \times \left( 
     1 + O\left(\frac{k}{(\log\log x)^2}\right) 
     \right). 
\end{equation}
\end{theorem}
\begin{proof}
We can extend the proofs in \cite[\S 7]{MV} to obtain 
analogous results on the distribution of $\omega(n)$. 
This result is cited as an exercise in \cite{MV}. 
\end{proof}

\section{The upper incomplete gamma function} 
\label{subSection_OtherFactsAndResults} 

\begin{definition}
The (upper) incomplete gamma function is defined by \cite[\S 8.4]{NISTHB} 
\[
\Gamma(a, z) = \int_{z}^{\infty} t^{a-1} e^{-t} dt, \mathtext{ for } 
     a, z \in \mathbb{R}^{+}.  
\]
The function $\Gamma(a, z)$ can be continued to an analytic function of $z$ on the 
universal covering of $\mathbb{C} \mathbin{\backslash} \{0\}$. 
We similarly define the regularized upper incomplete gamma function for real $a, z > 0$ by 
$Q(a, z) := \Gamma(a, z) \Gamma(a)^{-1}$. 
\end{definition}

\label{facts_ExpIntIncGammaFuncs} 
\begin{subequations}
The following properties are known \cite[\S 8.4; \S 8.11(i)]{NISTHB}: 
\begin{align} 
\label{eqn_IncompleteGamma_PropA} 
Q(a, z) & = e^{-z} \times \sum_{k=0}^{a-1} \frac{z^k}{k!}, \mathtext{ for } 
     a \in \mathbb{Z}^{+} \mathtext{ and } z \in \mathbb{R}^{+}, \\ 
\label{eqn_IncompleteGamma_PropB} 
\Gamma(a, z) & \sim z^{a-1} e^{-z}, \mathtext{ for fixed } a > 0 
     \mathtext{ and } z > 0 \mathtext{ as } z \rightarrow \infty. 
\end{align}
For $z > 0$, as $z \rightarrow \infty$ we have that \cite{NEMES2015C} 
\begin{equation} 
\label{eqn_IncompleteGamma_PropC}
\Gamma(z, z) = \sqrt{\frac{\pi}{2}} z^{z-\frac{1}{2}} e^{-z} \times \left(
	1 + O\left(\frac{1}{\sqrt{z}}\right)\right). 
\end{equation} 
For fixed, finite real $\rho > 0$, we define the sequence 
$\{b_n(\rho)\}_{n \geq 0}$ by the following recurrence relation: 
\[
b_n(\rho) = (1-\rho) \rho \cdot b_{n-1}^{\prime}(\rho) + 
	(2n-1) \rho \cdot b_{n-1}(\rho) + \delta_{n,0}. 
\]
For fixed $\rho > 0$, the sequence $\{b_n(\rho)\}_{n \geq 0}$ satisfies a 
Rodrigues type formula of the form \cite[Thm.~1.1]{NEMES2016}
\[
b_n(\rho) = (1-\rho)^n \times \frac{\partial^n}{\partial t^n}\left( 
     \frac{(\rho - 1) t}{\rho e^{t} - t - \rho}\right)^{n+1} \Biggr\rvert_{t=0}. 
\]
If $z,a \rightarrow \infty$ with $z = \rho a$ for some $\rho > 1$ such that 
$(\rho - 1)^{-1} = o\left(\sqrt{a}\right)$, then \cite{NEMES2015C}
\begin{equation}
\label{eqn_IncompleteGamma_PropD}
\Gamma(a, z) \sim z^a e^{-z} \times \sum_{n \geq 0} \frac{(-a)^n b_n(\rho)}{(z-a)^{2n+1}}. 
\end{equation} 
\end{subequations}

\begin{prop}
\label{prop_IncGammaLambdaTypeBounds_v1}
\begin{subequations}
Let $a,z > 0$ be taken such that as $a,z \rightarrow \infty$ (independently), the 
parameter $\rho := \frac{z}{a} > 0$ has a finite limit. 
The following results hold as $z \rightarrow \infty$: 
\begin{itemize}
\item
If $\rho \in (0, 1)$, then 
\begin{equation}
\Gamma(a, z) = \Gamma(a) + O_{\rho}\left(z^{a-1} e^{-z}\right). 
\end{equation}
\item 
If $\rho > 1$, then 
\begin{equation}
\Gamma(a, z) = \frac{z^{a-1} e^{-z}}{1-\rho^{-1}} + O_{\rho}\left(z^{a-2} e^{-z}\right). 
\end{equation}
\item
If $\rho > W(1) > 0.56714$, then 
\begin{equation}
\Gamma(a, z e^{\pm\pi\imath}) = -e^{\pm \pi\imath a} \frac{z^{a-1} e^{z}}{1 + \rho^{-1}} + 
     O_{\rho}\left(z^{a-2} e^{z}\right). 
\end{equation}
\end{itemize}
\end{subequations}
\end{prop}

\begin{remark*}
The first two estimates in 
Proposition \hlocalref{prop_IncGammaLambdaTypeBounds_v1} 
are only useful when $\rho$ is bounded away from the transition point at one. 
We cannot write the last expansion above 
as $\Gamma(a, -z)$ directly unless $a \in \mathbb{Z}^{+}$ 
as the incomplete gamma function 
has a branch point at the origin with respect to its second variable. 
This function becomes a single-valued 
analytic function of its second input by continuation 
on the universal covering of $\mathbb{C} \mathbin{\backslash} \{0\}$. 
\end{remark*}

\begin{proof}[Proof of Proposition \hlocalref{prop_IncGammaLambdaTypeBounds_v1}] 
The first asymptotic estimate follows directly from the following 
asymptotic series expansion that holds as $z \rightarrow \infty$ 
\cite[Eq.\ (2.1)]{NEMES2019}: 
\[
\Gamma(a, z) \sim \Gamma(a) + z^a e^{-z} \times \sum_{k \geq 0} 
     \frac{(-a)^k b_k(\rho)}{(z-a)^{2k+1}}. 
\]
Suppose that $\rho > 0$. 
The notation from \eqref{eqn_IncompleteGamma_PropD} and 
\cite[Thm.~1.1]{NEMES2016} shows that 
\[
\Gamma(a, z) = \frac{z^{a-1} e^{-z}}{1-\rho^{-1}} + z^{a} e^{-z} R_1(a, \rho). 
\]
From the bounds in \cite[\S 3.1]{NEMES2016}, we have 
\[
\left\lvert z^{a} e^{-z} R_1(a, \rho) \right\rvert \leq 
     z^a e^{-z} \times \frac{a \cdot b_1(\rho)}{(z-a)^{3}} = 
     \frac{z^{a-2} e^{-z}}{(1-\rho^{-1})^{3}}
\]
The main and error terms in the previous equation can also be 
seen by applying the asymptotic series in 
\eqref{eqn_IncompleteGamma_PropD} directly. 

The proof of the third equation above follows from the asymptotics 
\cite[Eq.\ (1.1)]{NEMES2015C}
\[
\Gamma(-a, z) \sim z^{-a} e^{-z} \times \sum_{n \geq 0} \frac{a^n b_n(-\rho)}{(z+a)^{2n+1}}, 
\]
by setting $(a, z) \mapsto \left(a e^{\pm \pi\imath}, z e^{\pm \pi\imath}\right)$ so that 
$\rho = \frac{z}{a} > W(1)$. 
The restriction on the range of $\rho$ over which the third formula holds is made to ensure that 
the formula from the reference is valid at negative real $a$. 
\end{proof}

\begin{lemma}
\label{lemma_ConvenientIncGammaFuncTypePartialSumAsymptotics_v2}
\begin{subequations}
\label{eqn_lemma_ConvenientIncGammaFuncTypePartialSumAsymptotics_v2}
As $x \rightarrow \infty$  
\begin{align}
\label{eqn_lemma_ConvenientIncGammaFuncTypePartialSumAsymptotics_v2_v1}
\left\lvert \sum_{1 \leq k \leq \log\log x} 
     \frac{(-1)^k (\log\log x)^{k-1}}{(k-1)!} \right\rvert 
     & = \frac{\log x}{2\sqrt{2\pi \log\log x}} \times 
     \left(1 + O\left(\frac{1}{\log\log x}\right)\right). 
\end{align}
For any $a \in \left(1, W(1)^{-1}\right) \subset (1, 1.76321)$ 
\begin{align}
\label{eqn_lemma_ConvenientIncGammaFuncTypePartialSumAsymptotics_v2_v2}
\left\lvert \sum_{k=1}^{a \log\log x} 
     \frac{(-1)^{k} (\log\log x)^{k-1}}{(k-1)!} 
     \right\rvert = 
     \frac{a^{\frac{1}{2} - \{a\log\log x\}}}{(1 + a)} 
     \times \frac{(\log x)^{a-a\log a}}{\sqrt{2\pi \log\log x}} \times 
     \left(1 + O\left(\frac{1}{\log\log x}\right)\right), 
     \mathtext{ as } x \rightarrow \infty. 
\end{align}
The function $\{x\} = x - \floor{x} \in [0, 1)$ 
denotes the fractional part of $x \in \mathbb{R}$.
\end{subequations}
\end{lemma}
\begin{proof}[Proof of Equation \eqref{eqn_lemma_ConvenientIncGammaFuncTypePartialSumAsymptotics_v2_v1}]
We have for $n \geq 1$ and any $t > 0$ by 
\eqref{eqn_IncompleteGamma_PropA} that 
\[
\sum_{1 \leq k \leq n} \frac{(-1)^k t^{k-1}}{(k-1)!} = -e^{-t} \times 
     \frac{\Gamma(n, -t)}{(n-1)!}. 
\]
Suppose that $t = n + \xi$ with $\xi = O(1)$. 
By the third formula 
in Proposition \hlocalref{prop_IncGammaLambdaTypeBounds_v1} 
with the parameters $(a, z, \rho) \mapsto \left(n, t, 1 + \frac{\xi}{n}\right)$, 
we deduce that as $n,t \rightarrow \infty$. 
\begin{equation*}
\Gamma(n, -t) = (-1)^{n+1} \times \frac{t^n e^{t}}{t+n} + 
     O\left(\frac{n t^n e^{t}}{(t+n)^3}\right) = 
     \frac{(-1)^{n+1} t^n e^t}{2n} + O\left(\frac{t^{n-1} e^t}{n}\right). 
\end{equation*}
Accordingly, we see that 
\[
\sum_{1 \leq k \leq n} \frac{(-1)^k t^{k-1}}{(k-1)!} = 
      \frac{(-1)^{n} t^n}{2n!} + O\left(\frac{t^{n-1}}{n!}\right). 
\]
The form of Stirling's formula in \cite[\cf Eq.\ (5.11.8)]{NISTHB} shows that 
\begin{align*}
n! & = \Gamma(1 + t - \xi) 
     = \sqrt{2\pi} t^{n+\frac{1}{2}} e^{-t} \times \left(1 + O\left(t^{-1}\right)\right). 
\end{align*}
Hence, as $n \rightarrow \infty$ with $t := n + \xi$ and $\xi = O(1)$, we obtain that 
\[
\sum_{k=1}^{n} \frac{(-1)^k t^{k-1}}{(k-1)!} = \frac{(-1)^n e^t}{2 \sqrt{2\pi t}} + 
     O\left(e^t t^{-\frac{3}{2}}\right). 
\]
The conclusion follows by taking $n := \floor{\log\log x}$ and $t := \log\log x$. 
\end{proof}
\begin{proof}[Proof of Equation \eqref{eqn_lemma_ConvenientIncGammaFuncTypePartialSumAsymptotics_v2_v2}]
The argument is nearly identical to the proof of the first equation. 
The key modifications are to set $t := an + \xi$ where $\xi = O(1)$, take the parameters 
$\left(a, z, \rho\right) \mapsto \left(an, t, 1 + \frac{\xi}{an}\right)$, and 
use the expansion $a^{an} = e^{an \log a}$, to simplify the main term obtained 
from Stirling's formula. 
\end{proof}

\section{Inversion of partial sums of Dirichlet convolutions}
\label{Section_PrelimProofs_Config} 
\label{subSection_PrelimProofs_Config_InversionTheorem}

\begin{proof}[Proof of Theorem \hlocalref{theorem_SummatoryFuncsOfDirCvls}] 
\label{proofOf_theorem_SummatoryFuncsOfDirCvls} 
Suppose that $h,r$ are arithmetic functions such that $r(1) \neq 0$. 
The following holds for all $x \geq 1$: 
\begin{align} 
\notag 
S_{r \ast h}(x) & := \sum_{n=1}^{x} \sum_{d|n} r(n) h\left(\frac{n}{d}\right), \\ 
\notag
     & \phantom{:} = 
     \sum_{d=1}^x r(d) \times H\left(\floor{\frac{x}{d}}\right), \\ 
\label{eqn_proof_tag_PigAsthx_ExactSummationFormula_exp_v2} 
     & \phantom{:} = 
     \sum_{i=1}^x \left(R\left(\floor{\frac{x}{i}}\right) - R\left(\floor{\frac{x}{i+1}}\right)\right) \times H(i). 
\end{align} 
The first formula for $S_{r \ast h}(x)$ is well known in the references. 
The second formula is justified directly using 
summation by parts as follows \cite[\S 2.10(ii)]{NISTHB}: 
\begin{align*} 
S_{r \ast h}(x) & = \sum_{d=1}^x h(d) \times R\left(\floor{\frac{x}{d}}\right), \\ 
     & = \sum_{i \leq x} \left(\sum_{j \leq i} h(j)\right) \times 
     \left(R\left(\floor{\frac{x}{i}}\right) - 
     R\left(\floor{\frac{x}{i+1}}\right)\right). 
\end{align*} 
For boolean-valued conditions \texttt{cond}, we adopt Iverson's convention that 
$\Iverson{\mathtt{cond}} \in \{0, 1\}$ evaluates to one precisely when 
\texttt{cond} is true and to zero otherwise.

We form the invertible matrix of coefficients (denoted by $\hat{R}$ below) 
by defining
\[
R_{x,j} := R\left(\Floor{x}{j}\right) \Iverson{j \leq x}, 
\]
and 
\[
r_{x,j} := R_{x,j} - R_{x,j+1}, \mathtext{ for } 1 \leq j \leq x. 
\] 
If we let $\hat{R} := (R_{x,j})$, then the next matrix is 
expressed by applying an invertible shift operation as 
\[
(r_{x,j}) = \hat{R} \left(I - U^{T}\right). 
\]
Since $r_{x,x} = R(1) = r(1) \neq 0$ for all $x \geq 1$ and $r_{x,j} = 0$ for all $j > x$, 
the matrix we have defined in this problem is lower triangular with a non-zero 
constant on its diagonals, and so is invertible. 

For any fixed $N \geq 1$, the $N \times N$ square matrix $U$ 
has $(i,j)^{th}$ entries for all $1 \leq i,j \leq N$ when $N \geq x$ that are defined by 
$(U)_{i,j} = \delta_{i+1,j}$ so that 
\[
\left[\left(I - U^T\right)^{-1}\right]_{i,j} = \Iverson{j \leq i}. 
\]
We observe that 
\[
\Floor{x}{j} - \Floor{x-1}{j} = \begin{cases} 
     1, & \text{ if $j|x$; } \\ 
     0, & \text{ otherwise. } 
     \end{cases} 
\] 
The previous equation implies that 
\begin{equation} 
\label{eqn_proof_tag_FloorFuncDiffsOfSummatoryFuncs_v2} 
R\left(\floor{\frac{x}{j}}\right) - R\left(\floor{\frac{x-1}{j}}\right) = 
     \begin{cases} 
     r\left(\frac{x}{j}\right), & \text{ if $j|x$; } \\ 
     0, & \text{ otherwise. } 
     \end{cases}
\end{equation} 
We use the property in \eqref{eqn_proof_tag_FloorFuncDiffsOfSummatoryFuncs_v2} 
to shift the matrix $\hat{R}$, and then invert the result to obtain a matrix involving the 
Dirichlet inverse of $r$, as follows: 
\begin{align*} 
\left(\left(I-U^{T}\right) \hat{R}\right)^{-1} & = 
     \left(r\left(\frac{x}{j}\right) \Iverson{j|x}\right)^{-1} = 
     \left(r^{-1}\left(\frac{x}{j}\right) \Iverson{j|x}\right). 
\end{align*} 
The target matrix is expressed by 
$$(r_{x,j}) = \left(I-U^{T}\right) \left(r\left(\frac{x}{j}\right) \Iverson{j|x}\right) \left(I-U^{T}\right)^{-1}.$$
We can evaluate its inverse by a similarity transformation conjugated by 
shift operators given by 
\begin{align*} 
(r_{x,j})^{-1} & = \left(I-U^{T}\right)^{-1} \left(r^{-1}\left(\frac{x}{j}\right) 
     \Iverson{j|x}\right) \left(I-U^{T}\right), \\ 
     & = \left(\sum_{k=1}^{\floor{\frac{x}{j}}} r^{-1}(k)\right) \left(I-U^{T}\right), \\ 
     & = \left(\sum_{k=1}^{\floor{\frac{x}{j}}} r^{-1}(k) - \sum_{k=1}^{\floor{\frac{x}{j+1}}} r^{-1}(k)\right). 
\end{align*} 
The summatory function $H(x)$ is expressed 
by a vector product with the inverse matrix from the previous equation as 
\begin{align*} 
H(x) & = \sum_{k=1}^x \left(\sum_{j=\floor{\frac{x}{k+1}}+1}^{\floor{\frac{x}{k}}} r^{-1}(j)\right) 
	\times S_{r \ast h}(k), \mathtext{ for } x \geq 1. 
\end{align*} 
We can prove another inversion formula by adapting our argument used to prove 
\eqref{eqn_proof_tag_PigAsthx_ExactSummationFormula_exp_v2} above. 
This leads to an alternate expression for $H(x)$ given by 
\[
H(x) = \sum_{k=1}^{x} r^{-1}(k) \times S_{r \ast h}\left(\Floor{x}{k}\right), 
     \mathtext{ for } x \geq 1. 
     \qedhere 
\]
\end{proof} 

\section{The proof of Theorem \hlocalref{lemma_HatCAstxSum_ExactFormulaWithError_v1}} 
\label{Appendix_ProofOfCOmegan_LogarithmicAvgOrderFormula}

\begin{lemma}
\label{lemma_eqn_proof_tag_SumLogCOmegan_P0_exp_v1}
As $x \rightarrow \infty$  
\begin{align}
\label{eqn_proof_tag_SumLogCOmegan_P0_exp_v1}
\sum_{n \leq x} \log C_{\Omega}(n) & = 
     \sum_{k \geq 1} \#\{n \leq x: \Omega(n)=k\} \times \log(k!) 
     \left(1 + O\left(\frac{1}{\sqrt{\log\log x}}\right)\right). 
\end{align}
\end{lemma}
\begin{proof}
Equation \eqref{eqn_proof_tag_hInvn_ExactNestedSumFormula_CombInterpetIdent_v3} shows that 
\[
\sum_{\substack{n \leq x \\ \mu^2(n)=1}} \log C_{\Omega}(n) = 
	\sum_{k \geq 1} \#\{n \leq x: \Omega(n)=k\} \times \log(k!). 
\]
The sum on the right-hand-side o the last equation is finite since 
$\Omega(n) \leq \log_2(x)$ for all $x \geq 2$. 
The key to the rest of the proof is to understand that the main term of the 
sum on the left-hand-side of 
equation \eqref{eqn_proof_tag_SumLogCOmegan_P0_exp_v1} 
is obtained by summing over only 
the squarefree $n \leq x$, i.e., the $n \leq x$ such that $\mu^2(n) = 1$. 
That is, we claim that 
\begin{equation}
\label{eqn_ProofTag_SqFreeSumSim_TargetClaim_v1}
\sum_{k \geq 1} \sum_{\substack{n \leq x \\ \Omega(n)=k}} \log C_{\Omega}(n) \sim 
	\sum_{k \geq 1} \sum_{\substack{n \leq x \\ \mu^2(n) = 1 \\ \Omega(n)=k}} \log C_{\Omega}(n). 
\end{equation}
The function $\operatorname{rad}(n)$ is the radix (or squarefree part) of $n$ which evaluates 
to the largest squarefree factor of $n$, 
or equivalently to the product of all primes $p | n$ 
\cite[\seqnum{A007913}]{OEIS}. 
For integers $x \geq 1$ and $1 \leq k \leq \log_2(x)$, define the sets 
\[
\mathcal{S}_k\left(\{\varpi_j\}_{j=1}^k; x\right) := \left\{2 \leq n \leq x: \mu(n) = 0, \omega(n) = k, 
	\frac{n}{\operatorname{rad}(n)} = p_1^{\varpi_1} \times \cdots \times p_k^{\varpi_k}, 
	\mathtext{ $p_i \neq p_j$ prime for $1 \leq i < j \leq k$}\right\}. 
\]
Recall that our goal is to show that the sums of $\log C_{\Omega}(n)$ at the non-squarefree $n \leq x$ 
corresponds to an error term on the right-hand-side of equation \eqref{eqn_ProofTag_SqFreeSumSim_TargetClaim_v1}. 
Then the idea behind the definition in the previous equation is as follows: 
For every non-squarefree integer $n \in [2, x]$, there is some $1 \leq k \leq \log_2(x)$ 
and a non-empty sequence of $k$ \emph{positive} integers $\{\varpi_j\}_{1 \leq j \leq k}$ such that 
$n \in \mathcal{S}_k\left(\{\varpi_j\}_{j=1}^k; x\right)$. 
Hence, we need to bound the growth of this function at non-empty sets 
$\{\varpi_j\}_{j=1}^k \subseteq \mathbb{Z}^{+}$ for $k \geq 1$. 

Let the function 
$$\mathcal{N}_k\left(\{\varpi_j\}_{j=1}^k; x\right) := 
  \frac{\left\lvert \mathcal{S}_k\left(\{\varpi_j\}_{j=1}^k; x\right) \right\rvert}{x}.$$ 
The special case where $\mathcal{W}_k := \{\varpi_j^{\ast}\}_{1 \leq j \leq k} \equiv \{0, 1\}$ 
is the set with unit value of multiplicity exactly one is denoted by 
$$\widehat{T}_k(x) := \mathcal{N}_k\left(\mathcal{W}_k; x\right).$$ 
If $n \in [2, x]$ is not squarefree and $n \in \mathcal{S}_k\left(\{\varpi_j\}_{j=1}^k; x\right)$, then 
we must have that $\varpi_j \geq 1$ for at least one $1 \leq j \leq k$. 
For any $k \geq 1$ and non-empty set 
$\{\varpi_j\}_{1 \leq j \leq k} \subset \mathbb{Z}^{+}$ of $k$ elements we have 
\[
\mathcal{N}_k\left(\{\varpi_j\}_{j=1}^k; x\right) \ll \widehat{T}_k(x). 
\]
\begin{subequations}
We claim that (proof below) 
\begin{equation}
\label{eqn_ProofTag_HatTkx_LL_v2}
     \widehat{T}_k(x) \ll \frac{1}{\sqrt{\log\log x}} \times 
     \#\{n \leq x: \omega(n)=k\}, 
     \mathtext{ for all } k \geq 1, \mathtext{ as } x \rightarrow \infty. 
\end{equation}
Suppose that equation \eqref{eqn_ProofTag_HatTkx_LL_v2} is true. 
Theorem \hlocalref{theorem_HatPi_ExtInTermsOfGz}, 
Theorem \hlocalref{remark_MV_Pikx_FuncResultsAnnotated_v1} and 
equation \eqref{eqn_ProofTag_HatTkx_LL_v2} imply that 
\begin{equation}
\label{eqn_ProofTag_HatTkx_LL_v2_v2}
\widehat{T}_k(x) \ll \frac{1}{\sqrt{\log\log x}} \times 
     \#\{n \leq x: \Omega(n)=k\}, 
     \mathtext{ for all } k \geq 1, \mathtext{ as } x \rightarrow \infty. 
\end{equation}
\end{subequations}
The upper bound on $\widehat{T}_k(x)$ in 
equation \eqref{eqn_ProofTag_HatTkx_LL_v2_v2} 
shows that the sum of denominator differences from 
\eqref{eqn_proof_tag_hInvn_ExactNestedSumFormula_CombInterpetIdent_v3} we subtract from the 
main term over the squarefree $n \leq x$ is asymptotically insubstantial. 
That is, we have proved that 
\[
\sum_{\substack{n \leq x \\ \mu(n)=0}} \log C_{\Omega}(n) = 
     o\left(\sum_{\substack{n \leq x \\ \mu^2(n)=1}} \log C_{\Omega}(n)\right), 
     \mathtext{ as } x \rightarrow \infty. 
     \qedhere
\]
\end{proof}

Recall the next two famous asymptotic formulae:
\begin{itemize}[noitemsep,topsep=0pt,leftmargin=0.64in]
\item[\textbf{1.}] As $x \rightarrow \infty$ 
     \cite[\S 22.4]{HARDYWRIGHT}
     $$\pi(x) = \frac{x}{\log x} \times \left(1 + O\left(\frac{1}{\log x}\right)\right).$$
\item[\textbf{2.}] A theorem of Mertens states that as $x \rightarrow \infty$ 
     \cite[\S 22.7--22.8]{HARDYWRIGHT}
     $$\sum_{p \leq x} p^{-1} \sim \log\log x.$$
\end{itemize}

\begin{proof}[Proof of Equation \eqref{eqn_ProofTag_HatTkx_LL_v2}]
The bound can be proved by induction on $k \geq 1$ using the inductive hypothesis
\[
\tag{IH}
\widehat{T}_m(x) \ll \frac{x^{1-2^{-m}} (\log\log x)^{m}}{(\log x)^{1+2^{-m}}}, 
	\mathtext{ for all } 1 \leq m \leq k, 
	\mathtext{ as } x \rightarrow \infty. 
\] 
The case where $k := 1$ is evaluated as follows: 
\begin{align*}
\widehat{T}_1(x) & = \sum_{p \leq \sqrt{x}} 1 \ll \frac{\sqrt{x}}{\log x}. 
\end{align*}
Suppose that $k \geq 1$ and that the (IH) holds at $k$. 
Theorem \hlocalref{remark_MV_Pikx_FuncResultsAnnotated_v1} 
and the bounds in the next equations show that 
equation \eqref{eqn_ProofTag_HatTkx_LL_v2} holds for all finite $k \geq 1$.
In particular, we have by the (IH) and 
H\"older's inequality with $\left(p^{-1}, q^{-1}\right) = \left(1-2^{-k}, 2^{-k}\right)$ that 
\begin{align*}
\widehat{T}_{k+1}(x) & \ll \sum_{p \leq \sqrt{x}} \widehat{T}_{k}\left(\frac{x}{p}\right) 
     \ll \frac{x^{1-2^{-k}} (\log\log x)^{k}}{(\log x)^{1+2^{-k}}} \times 
	\left(\sum_{p \leq x} p^{-1}\right)^{1-2^{-k}} \times 
	\pi\left(\sqrt{x}\right)^{2^{-k}} \\ 
     & \ll \frac{2^{2^{-k}} x^{1-2^{-(k+1)}} (\log\log x)^{k+1-2^{-k}}}{(\log x)^{1+2^{-(k-1)}}} ,
	\mathtext{ as } x \rightarrow \infty. 
\end{align*}
Thus, if the following equation holds, then the proof by induction is complete:
\begin{equation}
\label{eqn_ProofTag_HatTkx_LL_v3}
\widehat{U}_{k+1}(x) := \frac{2^{2^{-k}}}{(\log\log x)^{2^{-k}} (\log x)^{3 \cdot 2^{-(k+1)}}} \ll 1, 
     \mathtext{ as } x \rightarrow \infty. 
\end{equation}
For any fixed finite large $x$, we have that $1 \leq k \leq \log_2(x)$. 
In particular, we have that $\frac{1}{x} \leq 2^{-k} \leq \frac{1}{2}$ for all $k$ and large $x$. 
We can then establish the following to prove that 
equation \eqref{eqn_ProofTag_HatTkx_LL_v3} holds: 
\begin{align*}
\widehat{U}_{k+1}(x) & \ll \exp\left(-\frac{3 \log\log\log x}{2x}\right) = 
     1 + O\left(\frac{\log\log\log x}{x}\right). 
\end{align*}
The proof of equation \eqref{eqn_ProofTag_HatTkx_LL_v2} 
is completed by applying 
Theorem \hlocalref{remark_MV_Pikx_FuncResultsAnnotated_v1}. 
\end{proof}

\begin{proof}[Proof of Theorem \hlocalref{lemma_HatCAstxSum_ExactFormulaWithError_v1}]  
We will split the full sum on the left-hand-side of 
\eqref{eqn_proof_tag_SumLogCOmegan_P0_exp_v1} into two sums, 
each over disjoint indices, that form the main and error terms, 
$L_{\Omega}(x)$ and $\widehat{L}_{\Omega}(x)$, respectively.
For $x \geq 19$, consider the following partial sums\footnote{
     Note that the choice of the constant $\frac{3}{2}$ is arbitrary. 
     We can obtain the same result by splitting the upper (lower) bound 
     on $\Omega(n)$ at $r\log\log x$ for any fixed $r > 1$. 
}:
\begin{align*}
L_{M,\Omega}(x) & := 
     \sum_{\substack{n \leq x \\ \Omega(n) \leq \frac{3}{2} \log\log x}} \log C_{\Omega}(n), \\ 
L_{E,\Omega}(x) & := 
     \sum_{\substack{n \leq x \\ \Omega(n) > \frac{3}{2} \log\log x}} 
	\log C_{\Omega}(n).
\end{align*}
\begin{subequations}
We claim that the main term is given by 
\begin{equation}
\label{eqn_proof_tag_LOmegax_MainTerm_v1}
L_{M,\Omega}(x) = 
	x (\log\log x)(\log\log\log x) \left(1 + 
     O\left(\frac{1}{\sqrt{\log\log x}}\right)\right).
\end{equation}
To bound the error term, we claim that 
\begin{equation}
\label{eqn_proof_tag_PartialSumsOver_HatCkx_EquivCond_v2}
L_{E,\Omega}(x) = o\left(x \sqrt{\log\log x} (\log\log\log x)\right), 
     \mathtext{ as } x \rightarrow \infty. 
\end{equation}
\end{subequations}
The proofs of equations \eqref{eqn_proof_tag_LOmegax_MainTerm_v1} and 
\eqref{eqn_proof_tag_PartialSumsOver_HatCkx_EquivCond_v2} are completed below. 
Equations \eqref{eqn_proof_tag_LOmegax_MainTerm_v1} and 
\eqref{eqn_proof_tag_PartialSumsOver_HatCkx_EquivCond_v2} 
yield the conclusion of 
Theorem \hlocalref{lemma_HatCAstxSum_ExactFormulaWithError_v1} since 
\[
\sum_{n \leq x} \log C_{\Omega}(n) = L_{M,\Omega}(x) + L_{E,\Omega}(x), 
     \mathtext{ for all } x > e^e. 
     \qedhere
\]
\end{proof}

\begin{proof}[Proof of Equation \eqref{eqn_proof_tag_LOmegax_MainTerm_v1}]
Lemma \hlocalref{lemma_eqn_proof_tag_SumLogCOmegan_P0_exp_v1} and 
Theorem \hlocalref{theorem_HatPi_ExtInTermsOfGz} from the appendix show that 
\[
L_{M,\Omega}(x) = \frac{x}{\log x} \times \sum_{1 \leq k \leq \frac{3}{2} \log\log x} 
	\mathcal{G}\left(\frac{k-1}{\log\log x}\right) \frac{(\log\log x)^{k-1}}{(k-1)!} 
	\log(k!) \times \left(1 + O\left(\frac{1}{\log\log x}\right)\right), 
\]
where 
\[
\mathcal{G}(z) := \frac{1}{\Gamma(1+z)} \times \prod_p \left(1 - \frac{z}{p}\right)^{-1} \left( 
	1 - \frac{1}{p}\right)^z, 
	\mathtext{ for } |z| < 2. 
\]
For any $j \geq 0$, 
Binet's formula for the log-gamma function is stated as 
\cite[\S 5.9(i)]{NISTHB} 
\[
\log j! = \left(j+\frac{1}{2}\right)\log(1+j) - j + O(1), 
\]
where $z! := \Gamma(1 + z)$ for any real-valued $z \geq 0$. 
Let the function 
$$g(x, k) := \frac{(\log\log x)^{k-1}}{2 (k-1)!} \times \left((2k+1) \log(1+k) - 2k + O(1)\right).$$ 
Binet's formula shows that 
\[
L_{M,\Omega}(x) = \frac{x}{\log x} \times \sum_{1 \leq k \leq \frac{3}{2} \log\log x} 
	\mathcal{G}\left(\frac{k-1}{\log\log x}\right) g(x, k) \times 
	\left(1 + O\left(\frac{1}{\log\log x}\right)\right). 
\]
The Euler-Maclaurin summation (EM) formula \cite[\S 9.5]{GKP} shows that 
for each fixed integer $m \geq 1$ and function $f(t)$ that is $m$ times 
continuously differentiable on $(0, \infty)$ 
\cite[\seqnum{A000367}; \seqnum{A002445}]{OEIS} 
\begin{align*}
L_{M,\Omega}(x) & = \frac{x}{\log x} \times \Biggl(
     \int_1^{\frac{3}{2} \log\log x} 
	\mathcal{G}\left(\frac{t-1}{\log\log x}\right) g(x, t) dt + 
     \frac{1}{2}\mathcal{G}\left(\frac{3}{2}-\frac{1}{\log\log x}\right) 
     g\left(x, \frac{3}{2} \log\log x\right) - \frac{g(x, 1)}{2} \\ 
	& \phantom{=\frac{x}{\log x} \times\ } + 
     O\left(\sum_{k=1}^{m} \frac{B_{k}}{k!} \times \frac{\partial^{(k-1)}}{\partial t^{(k-1)}} 
	\left[\mathcal{G}\left(\frac{t-1}{\log\log x}\right) g(x, t) 
     \right]_{t=1}^{t=\frac{3}{2}\log\log x} + \widehat{R}_{m}[\mathcal{G} \cdot g]\right) 
	\Biggr) \times \left(1 + O\left(\frac{1}{\log\log x}\right)\right). 
\end{align*}
The degree-$m$ remainder term in the previous equation is bounded by 
\begin{align*}
\left\lvert \widehat{R}_m[\mathcal{G} \cdot g] \right\rvert & = 
     O\left(\frac{1}{m!} \times \int_1^{\frac{3}{2} \log\log x} B_m(\{t\}) f^{(m)}(t) dt\right). 
\end{align*}
When $f(t) \equiv f_x(t) := \mathcal{G}\left(\frac{t-1}{\log\log x}\right) g(x, t)$, 
we denote the degree-$m$ EM formula error term by 
$$E_m(x) := \sum\limits_{k=1}^{m} \frac{B_k}{k!} \times \frac{\partial^{(k-1)}}{\partial t^{(k-1)}} 
	\left[\mathcal{G}\left(\frac{t-1}{\log\log x}\right) g(x, t) 
     \right]_{t=1}^{t=\frac{3}{2} \log\log x} + \frac{1}{m!} \times 
     \int_1^{\frac{3}{2} \log\log x} \left\lvert B_m(\{t\}) \right\rvert f_x^{(m)}(t) dt$$ 
\begin{subequations}
It suffices to choose $m := 1$ in the expansion above. 
Specializing to the case where $m := 1$ yields the upper bound 
\begin{align}
\label{eqn_ProofTag_MainTermFunc_LMOmegax_v1_EMErrorTerm_v1}
\left\lvert E_1(x) \right\rvert & \ll \frac{(\log x) (\log\log\log x)}{\sqrt{\log\log x}} + 
     \underset{:= I_1(x)}{\underbrace{\int_1^{\frac{3}{2} \log\log x} t^2 \log(1+t) 
     \times \frac{(\log\log x)^t}{(t+1)!} dt}}, 
     \mathtext{ as } x \rightarrow \infty. 
\end{align}
The integral term, $I_1(x)$, in equation 
\eqref{eqn_ProofTag_MainTermFunc_LMOmegax_v1_EMErrorTerm_v1} is 
bounded using H\"older's inequality with $(p, q) := (1, \infty)$ as follows: 
\begin{align}
\notag
\left\lvert I_1(x) \right\rvert & \ll (\log\log x) \times \max_{1 \leq t \leq \frac{3}{2} \log\log x} 
     \frac{(\log\log\log x) (\log\log x)^{t+2}}{(1 + t) \Gamma(1 + t)} \\ 
\notag 
     & \ll (\log\log x)^3 (\log\log\log x) \times \max_{1 \leq t \leq \frac{3}{2} \log\log x} 
     \frac{(\log\log x)^t e^t}{t^{t+\frac{3}{2}}}, 
     \mathtext{ as } x \rightarrow \infty. 
\end{align}
The maximum in the previous equations is attained when $t := 1$. 
This shows that 
\begin{align}
\notag 
\frac{x}{\log x} \times |E_1(x)| & \ll \frac{x}{\log x} \times |I_1(x)| 
     \ll 
     \frac{x (\log\log x)^{4} (\log\log\log x)}{\log x} \\  
\label{eqn_ProofTag_MainTermFunc_LMOmegax_v1_EMErrorTerm_v3}
     & = 
     o\left(x \sqrt{\log\log x} (\log\log\log x)\right), 
     \mathtext{ as } x \rightarrow \infty. 
\end{align}
\end{subequations}
\begin{subequations}
\label{eqn_ProofTag_MainTermFunc_LMOmegax_v2}
\noindent
The mean value theorem states that for each large $x$ there is some 
$c \in \left[1, \log\log x\right]$ such that 
\begin{align}
\label{eqn_ProofTag_MainTermFunc_LMOmegax_v2.1} 
     \int_1^{\frac{3}{2} \log\log x} 
	\mathcal{G}\left(\frac{t-1}{\log\log x}\right) g(x, t) dt & = 
	\mathcal{G}\left(\frac{c-1}{\log\log x}\right) \times 
     \int_1^{\frac{3}{2} \log\log x} g(x, t) dt. 
\end{align}
For any real $y > e^e$, let 
\begin{equation}
\label{eqn_ProofTag_MainTermFunc_LMOmegax_v2.2}
c(y) := \inf \left\{c \in \left[1, \frac{3}{2} \log\log y\right]: 
     \mathtext{ equation \eqref{eqn_ProofTag_MainTermFunc_LMOmegax_v2.1} holds}\right\}. 
\end{equation}
\end{subequations}
Let $B_0^{\ast}(x) := \mathcal{G}\left(\frac{c(x)-1}{\log\log x}\right)$. 
We can apply the EM formula again to see that 
\begin{align}
\label{eqn_ProofTag_MainTermFunc_LMOmegax_v3} 
 & L_{M,\Omega}(x) = \frac{x}{\log x} \times \Biggl(
     \sum_{1 \leq k \leq \frac{3}{2} \log\log x} B_0^{\ast}(x) g(x, k) + 
	\frac{1}{2}\left(1 - B_0^{\ast}(x)\right) 
     \mathcal{G}\left(\frac{3}{2}-\frac{1}{\log\log x}\right) 
     g\left(x, \frac{3}{2} \log\log x\right) \\ 
\notag
	& \phantom{=\quad\ } + 
     O\left(1 + \sum_{k=1}^{m} \frac{B_{k}}{k!} \times \frac{\partial^{(k-1)}}{\partial t^{(k-1)}} 
	\left[\left(1 + \mathcal{G}\left(\frac{t-1}{\log\log x}\right)\right) g(x, t) 
     \right]_{t=1}^{t=\frac{3}{2} \log\log x} + \widehat{R}_m[(1 + \mathcal{G}) \cdot g]\right) 
	\Biggr) \times \left(1 + O\left(\frac{1}{\log\log x}\right)\right). 
\end{align}
For $m := 1$, the EM formula error term 
satisfies $\left\lvert \widehat{R}_m[(1 + \mathcal{G}) \cdot g] \right\rvert \ll \left\lvert I_1(x) \right\rvert$ 
for all sufficiently large large $x$. 
Equation \eqref{eqn_ProofTag_MainTermFunc_LMOmegax_v1_EMErrorTerm_v3} 
provides bounds on $|E_1(x)|$ and $|I_1(x)|$. 

We have two remaining steps to establish equation \eqref{eqn_proof_tag_LOmegax_MainTerm_v1}: 
\begin{itemize}[leftmargin=0.5in]
\item[\textbf{(i)}] To show that the sums on the right-hand-side of equation 
     \eqref{eqn_ProofTag_MainTermFunc_LMOmegax_v3} give the main term of this expression for 
     $L_{M,\Omega}(x)$ (up to a factor of the bounded function $B_0^{\ast}(x)$); and 
\item[\textbf{(ii)}] To show that there is a limiting constant with 
     $B_0^{\ast}(x) \xrightarrow{x \rightarrow \infty} 1$. 
\end{itemize}
The sums in the previous equation are approximated using Abel summation applied to the 
following functions for $1 \leq u \leq \log\log x$: 
\[
A_x(u) := \sum_{1 \leq k \leq u} \frac{x (\log\log x)^{k-1}}{(\log x) (k-1)!} = 
     \frac{x \Gamma\left(u, \log\log x\right)}{\Gamma\left(u\right)}; 
     \mathtext{ and } 
     f(u) := \frac{(2u+1)}{2} \log\left(1 + u\right) - u + O(1). 
\]
That is, we have by Proposition \hlocalref{prop_IncGammaLambdaTypeBounds_v1} that 
\begin{align}
\notag
\frac{x}{\log x} \times & \sum_{1 \leq k \leq \log\log x} g(x, k) \\ 
\notag
     & = A_x\left(\frac{3}{2} \log\log x\right) f\left(\frac{3}{2} \log\log x\right) - 
     (\log\log x) \times 
     \bigintssss_{\frac{1}{\log\log x}}^{\frac{3}{2}} 
     A_x(\alpha \log\log x) f^{\prime}(\alpha \log\log x) d\alpha + 
     O\left(\frac{x}{\log x}\right) \\ 
\notag 
     & = 
     \frac{3x}{2} (\log\log x) (\log\log\log x) \left(1 + 
     O\left(\frac{1}{\log\log x}\right)\right) + 
     \bigintssss_1^{\frac{3}{2}} f^{\prime}(\alpha \log\log x) \left(1 + 
     O\left(\frac{1}{\sqrt{\log\log x}}\right)\right) d\alpha \\ 
\notag
     & \phantom{\frac{3x}{2} (\log\log x) (\log\log\log x)\Biggl(1 + \ \ } + 
     O\left(x \sqrt{\log\log x} \times \bigintssss_0^1 f^{\prime}(\alpha \log\log x) \times 
     (\log x)^{\alpha - 1} d\alpha\right) \\
\label{eqn_proof_tag_LOmegax_MainTerm_v2}
	& = 
	x (\log\log x)(\log\log\log x) \left(1 + 
     O\left(\frac{1}{\sqrt{\log\log x}}\right)\right). 
\end{align}
\begin{subequations}
\label{eqn_ProofTag_MiscObservations_EMMVT_v1}
We have the following observations: 
\begin{align}
\mathcal{G}\left(\frac{3}{2}-\frac{1}{\log\log x}\right) & = \mathcal{G}\left(\frac{3}{2}\right) 
	\left(1 + O\left(\frac{1}{\log\log x}\right)\right), 
	\mathtext{ as } x \rightarrow \infty, \\ 
g\left(x, \frac{3}{2} \log\log x\right) & \ll 
     (\log x)^{\frac{3}{2}\left(1 - \log\left(\frac{3}{2}\right)\right)} 
     \sqrt{\log\log x} (\log\log\log x) \\ 
\notag
     & = 
     o\left((\log x) \sqrt{\log\log x} (\log\log\log x)\right), 
	\mathtext{ as } x \rightarrow \infty. 
\end{align}
\end{subequations}
Equations \eqref{eqn_ProofTag_MainTermFunc_LMOmegax_v3} and 
\eqref{eqn_proof_tag_LOmegax_MainTerm_v2} and 
\eqref{eqn_ProofTag_MiscObservations_EMMVT_v1} show that 
\begin{align}
\label{eqn_ProofTag_MainTermFunc_LMOmegax_v4} 
L_{M,\Omega}(x) & = B_0^{\ast}(x) x (\log\log x) (\log\log\log x) 
     \left(1 + O\left(\frac{1}{\sqrt{\log\log x}}\right)\right), 
	\mathtext{ as } x \rightarrow \infty. 
\end{align}
This observation accomplishes step \textbf{(i)}. 

Let $\mathcal{C}(t) := \mathcal{G}\left(\frac{c(t)-1}{\log\log \left\lfloor t \right\rfloor}\right)$ where 
$c(t)$ is defined as in equation \eqref{eqn_ProofTag_MainTermFunc_LMOmegax_v2.2} 
for all real $t \geq 19$. 
It is not difficult to see that $\mathcal{C}(t)$ is continuous and differentiable at all 
$t \in \left(e^e, \infty\right)$. 
We can see by computation from equation \eqref{eqn_ProofTag_MainTermFunc_LMOmegax_v2.1} that 
for all sufficiently large $t > e^e$, the derivative of this function satisfies 
$\mathcal{C}^{\prime}(t) < 0$. 
Moreover, we have that $1 \leq \mathcal{C}(t) < 2$ for all $t \geq \ceiling{e^e}$ where 
$0 \leq \frac{c(x)-1}{\log\log x} \leq 1$ for all integers $x \geq 19$. 
This means that 
\[
\lim_{x \rightarrow \infty} \mathcal{C}(x) = \mathcal{G}(0) = 1, 
\]
and so $B_0^{\ast}(x) \xrightarrow{x \rightarrow \infty} 1$. 
This completes step \textbf{(ii)}. 
Step \textbf{(ii)} combined with equation \eqref{eqn_ProofTag_MainTermFunc_LMOmegax_v4} 
(e.g., with step \textbf{(i)}) 
shows that equation \eqref{eqn_proof_tag_LOmegax_MainTerm_v1} holds. 
\end{proof}

\begin{proof}[Proof of Equation \eqref{eqn_proof_tag_PartialSumsOver_HatCkx_EquivCond_v2}]
The following equation holds:
\begin{equation}
\label{eqn_proof_tag_LogCOmegan_ll_FuncOfOmegan_v1}
\log C_{\Omega}(n) \ll \Omega(n) \log \Omega(n), \mathtext{ for } n \leq x, 
     \mathtext{ as } x \rightarrow \infty. 
\end{equation}
The bound for $\log C_{\Omega}(n)$ in 
\eqref{eqn_proof_tag_LogCOmegan_ll_FuncOfOmegan_v1} stated in terms of the variable 
$n \leq x$ holds as the upper bound on the interval $x \rightarrow \infty$. 
The right-hand-side terms involving $\Omega(n) \in [1, \log_2(x)]$ 
oscillate in magnitude over the integers $1 \leq n \leq x$. 
The statement in the last equation follows by 
maximizing (and minimizing) the ratio of the right-hand-side of 
\eqref{eqn_proof_tag_LogCOmegan_ll_FuncOfOmegan_v1} to 
Binet's log-gamma formula. 
We use the following notation to express this ratio as 
\[
\mathcal{R}(n, j) := \frac{\Omega(n) \log \Omega(n)}{\left(j + \frac{1}{2}\right) \log(1 + j) - j}. 
\]
Numerical methods show that $\mathcal{R}(n, \Omega(n))$ is absolutely bounded for all 
$16 \leq n \leq x$ as $x \rightarrow \infty$. 
The global extrema of this function on the positive integers 
are each attained as $\mathcal{R}(n_{\ell}, j_{\ell}), \mathcal{R}(n_{u}, j_{u})$ for 
finite integers $2 \leq n_{\ell}, n_u < +\infty$ and 
$j_{\ell},j_u = \Omega(n_{\ell}), \Omega(n_u) \in [1, 12]$. 

The analog of the Erd\H{o}s-Kac theorem for the function $\Omega(n)$ is given by 
\cite[Thm.~7.21; \S 7.4]{MV} 
\[
\frac{1}{x} \times \#\left\{3 \leq n \leq x: \frac{\Omega(n) - \log\log x}{\sqrt{\log\log x}} \leq z\right\} = 
     \Phi(z) + O\left(\frac{1}{\sqrt{\log\log x}}\right), \mathtext{ for } z \in (-\infty, \infty). 
\]
Then we have that as $x \rightarrow \infty$ 
\begin{align*}
L_{E,\Omega}(x) & \ll \sum_{\substack{n \leq x \\ \Omega(n) \geq \frac{3}{2} \log\log x}} 
	\Omega(n) \log \Omega(n) 
     \ll \left(\sum_{\substack{n \leq x \\ \Omega(n) \geq \frac{3}{2} \log\log x}} 
	\Omega(n)^2\right)^{\frac{1}{2}} \times 
     \left(\sum_{\substack{n \leq x \\ \Omega(n) \geq \frac{3}{2} \log\log x}} 
     \log^2 \Omega(n) \right)^{\frac{1}{2}} \\ 
     & \ll (\log\log x) \times 
	\left(\frac{x}{\sqrt{\log\log x}} \times \bigintsss_{\log\log x}^{\log_2(x)} t^2 
	e^{-\frac{\left(t-\log\log x\right)^2}{2\log\log x}} dt\right)^{\frac{1}{2}} \times 
     \#\left\{n \leq x: \Omega(n) \geq \frac{3}{2} \log\log x\right\}^{\frac{1}{2}}.
\end{align*} 
The change of variable $u := \frac{t-\log\log x}{\sqrt{\log\log x}}$ and 
Theorem \hlocalref{theorem_MV_Thm7.20-init_stmt} applied to the previous equation 
show that 
\begin{align*}
L_{E,\Omega}(x) & \ll 
     x (\log\log x)^2 (\log x)^{\frac{1}{4}\left(1 - 3\log\left(\frac{3}{2}\right)\right)} 
     \ll x (\log\log x)^2 (\log x)^{-0.0540987}. 
     \qedhere
\end{align*}
\end{proof} 

\end{document}